\documentclass[11pt]{amsart} 

\usepackage{color}
\usepackage{amsthm}
\usepackage{graphicx}
\usepackage{amssymb}
\usepackage[mathscr]{eucal}
\usepackage{amsfonts}
\usepackage{amsmath}
\usepackage{amsthm}
\usepackage{amssymb}
\usepackage{latexsym}

\usepackage[normalem]{ulem}

\usepackage[
paper=letterpaper,margin=1in]{geometry}
\newgeometry{top=1in,bottom=1in,right=1in,left=1in}

\numberwithin{equation}{section}

\newcommand{\bc}{\begin{center}}
\newcommand{\ec}{\end{center}}

\newcommand{\be}{\begin{equation}}
\newcommand{\ee}{\end{equation}}

\newcommand{\beqn}{\begin{eqnarray*}}
\newcommand{\eeqn}{\end{eqnarray*}}

\newcommand{\gl}{\lambda}

\newcommand{\x}{\times}

\newcommand{\daha}{\ddot H}

\def\SH{\mathrm {S} \daha}
\newcommand{\e}{{\mathbf{e} }}

\def\E{\mathcal{E}}

\def\gl{\mathfrak{gl}}

\newtheorem{theorem}{Theorem}[section]
\newtheorem*{theorem*}{Theorem}
\newtheorem{proposition}[theorem]{Proposition}
\newtheorem{corollary}[theorem]{Corollary}
\newtheorem{lemma}[theorem]{Lemma}
\theoremstyle{definition}
\newtheorem{definition}[theorem]{Definition}
\newtheorem{remark}[theorem]{Remark}

\newtheorem{conjecture}[theorem]{Conjecture}

\newcommand{\xx}{{\mathbf{x}}}
\newcommand{\yy}{{\mathbf{y}}}
\newcommand{\zz}{{\mathbf{z}}}
\newcommand{\cc}{{\mathbf{c}}}
\newcommand{\im}{{\mathrm{im}}}
\def\Z{\mathbb{Z}}
\def\CC{\mathcal{C}}

\usepackage{pinlabel}
\newcommand{\pic}[2]{\raisebox{-.5\height}{\includegraphics[scale=#2]{#1}}}

\newcommand\Xor{\pic{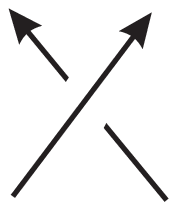}{.50}}
\newcommand\Yor{\pic{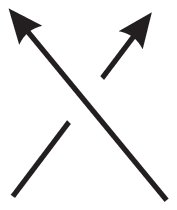} {.50}}
\newcommand\Ior{\pic{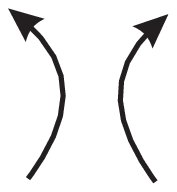} {.50}}
\newcommand\Torusbase{\pic{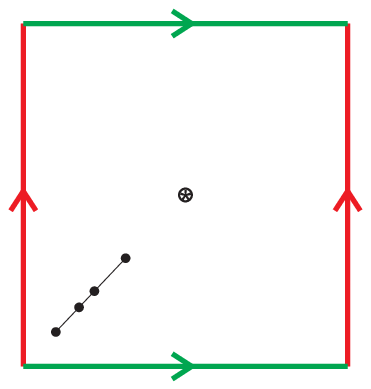} {1}}
\newcommand\Torusbasedisc{\pic{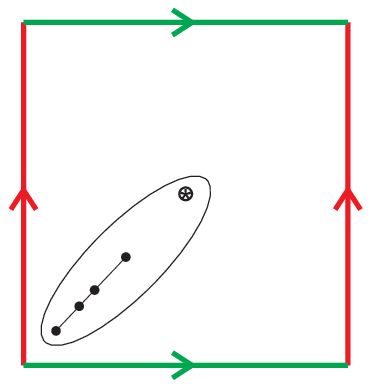} {1}}
\newcommand\xiibraid{\pic{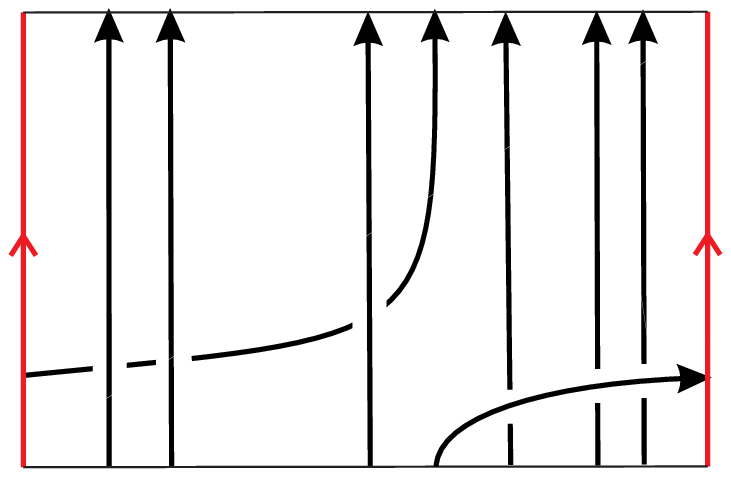}{.7}}
\newcommand\xiibraidblue{\pic{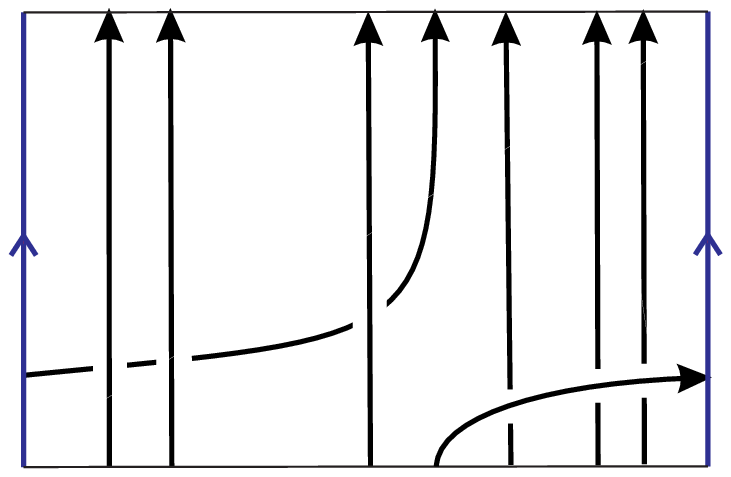}{.6}}
\newcommand\xii{\pic{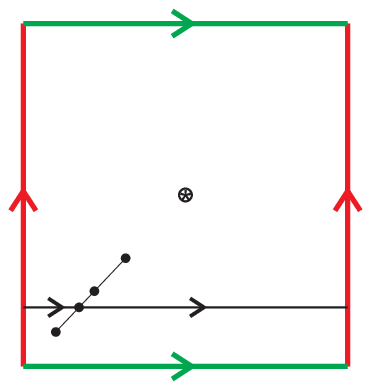}{1}}
\newcommand\etai{\pic{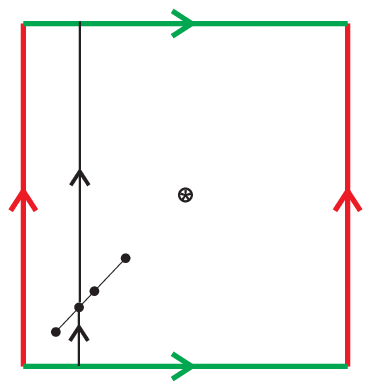}{1}}
\newcommand\etaibraid{\pic{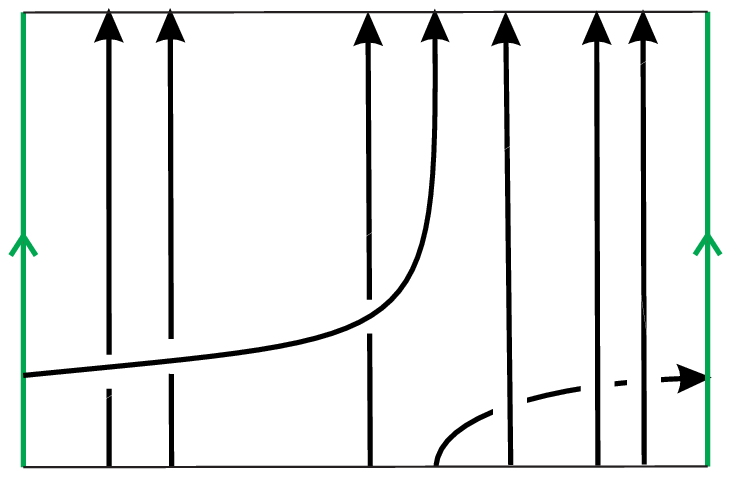}{.7}}
\newcommand\etaibraidblue{\pic{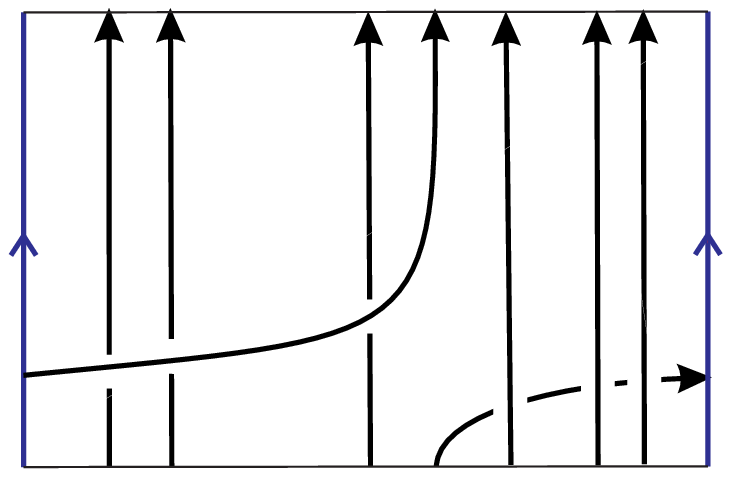}{.6}}
\newcommand\etatwoxione{\pic{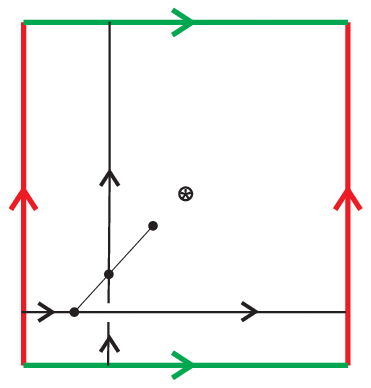}{1}}
\newcommand\xioneetatwo{\pic{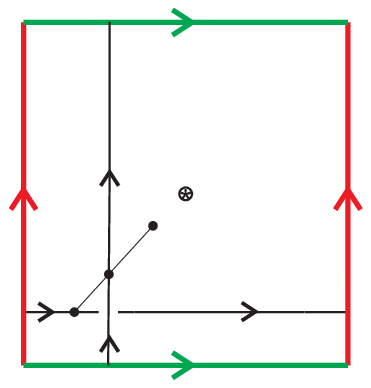}{1}}
\newcommand\xioneetaone{\pic{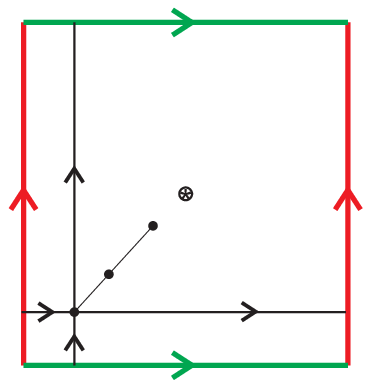}{1}}
\newcommand\xioneetaonedivert{\pic{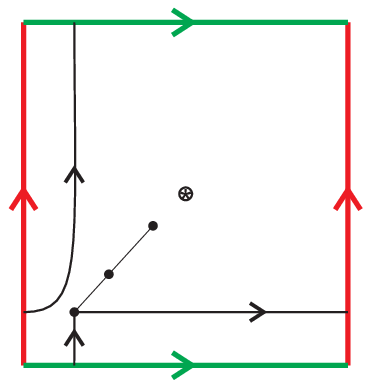}{1}}
\newcommand\etaonexionedivert{\pic{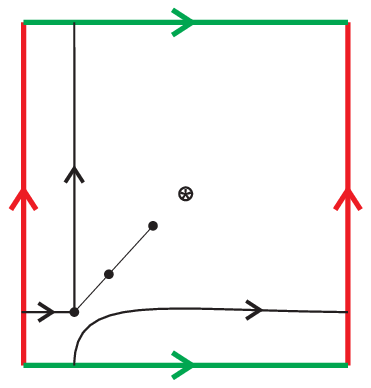}{1}}
\newcommand\xioneetaonecommutator{\pic{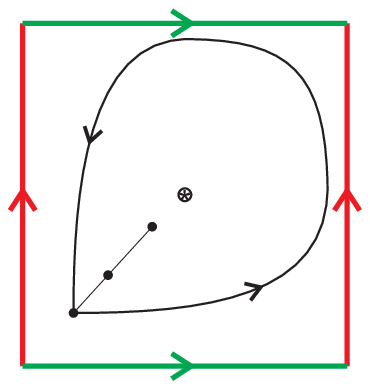}{1}}
\newcommand\xioneetaonecommutatorbraid{\pic{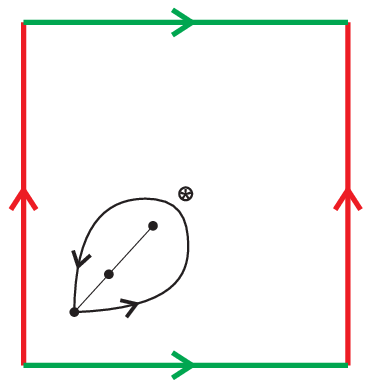}{1}}

\newcommand\etaonexitwo{\pic{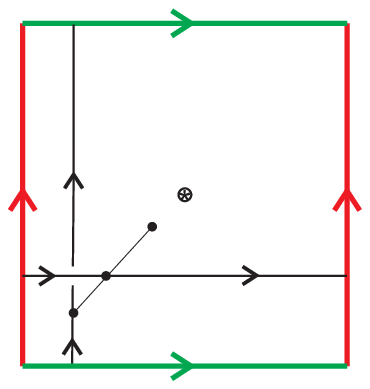}{1}}
\newcommand\xitwoalphatwo{\pic{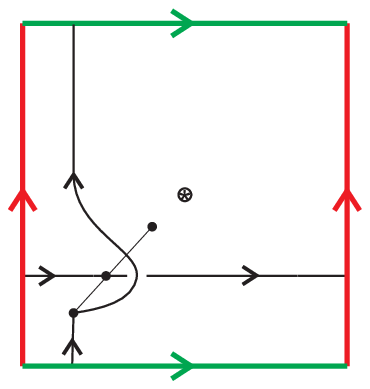}{1}}
\newcommand\alphatwo{\pic{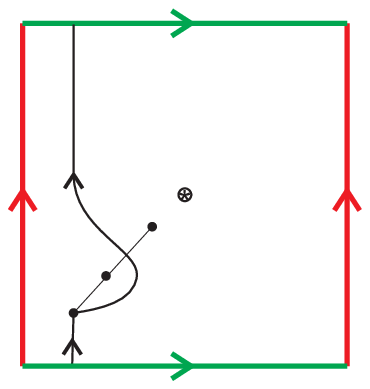}{1}}
\newcommand\alphan{\pic{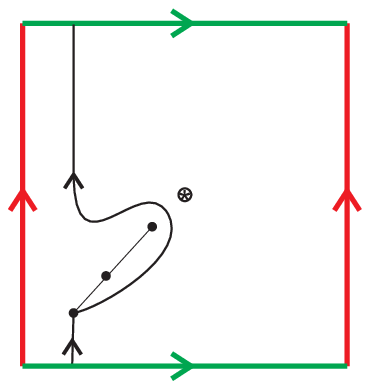}{1}}
\newcommand\etaonexitwon{\pic{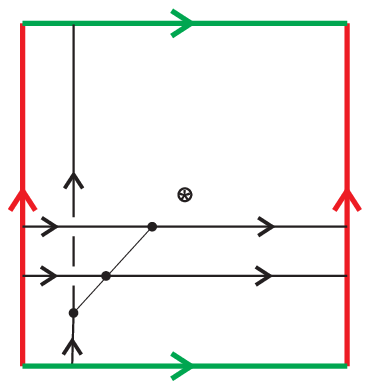}{1}}
\newcommand\etaonexitwondivert{\pic{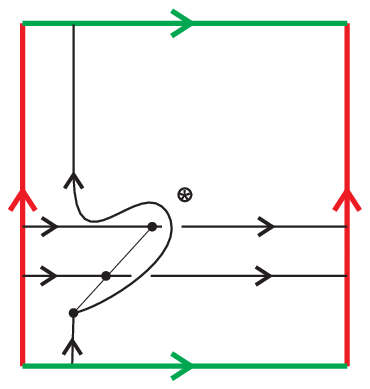}{1}}

\newcommand\betanbraid{\pic{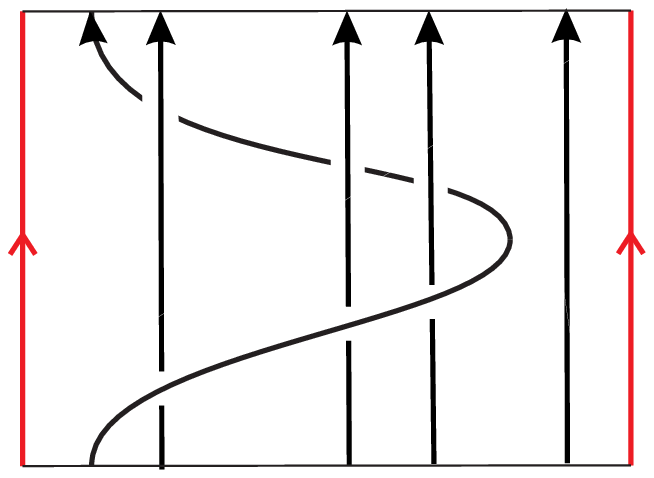}{.7}}

\newcommand\sigmaiplan{\pic{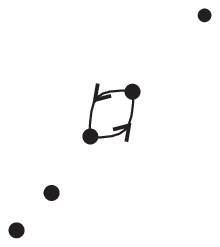}{1.2}}

 \newcommand\backidblue{\pic{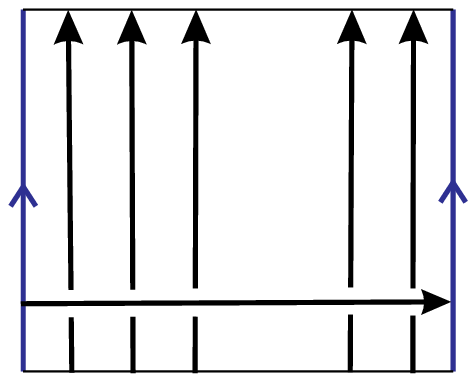}{.7}}
 \newcommand\frontidblue{\pic{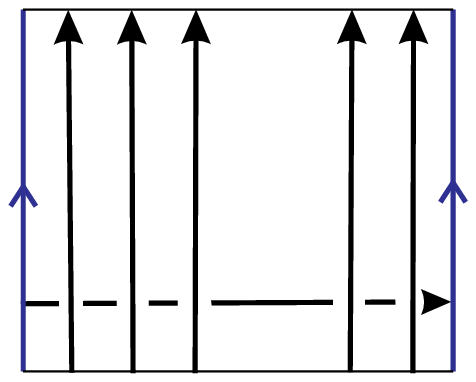}{.7}}
 \newcommand\backidleftblue{\pic{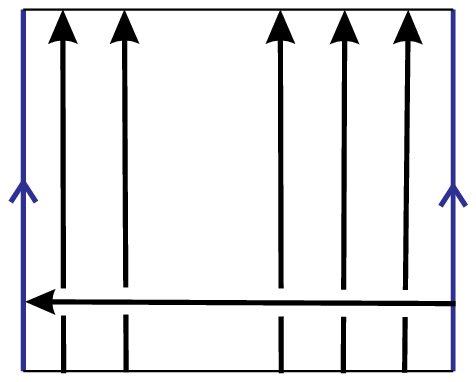}{.7}}
 \newcommand\frontidleftblue{\pic{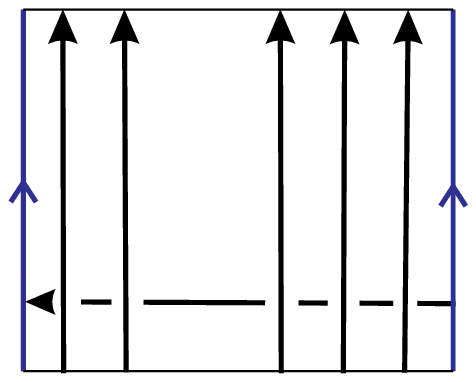}{.7}}
 
 \newcommand\braidplanover{\pic{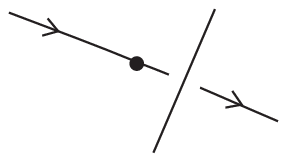}{1.2}}
 \newcommand\braidplanunder{\pic{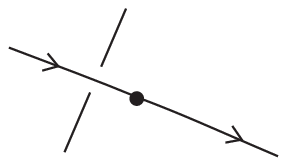}{1.2}}
  \newcommand\uxfront{\pic{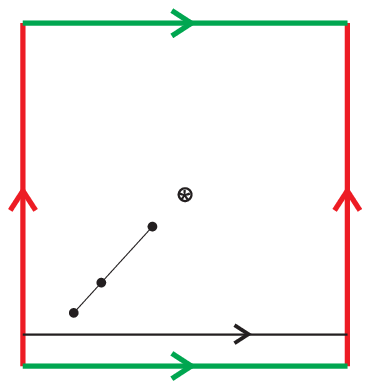}{1}}
 \newcommand\uxback{\pic{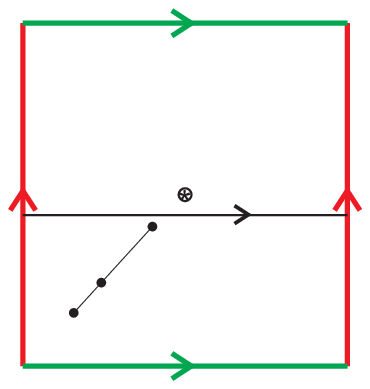}{1}}
 \newcommand\uyfront{\pic{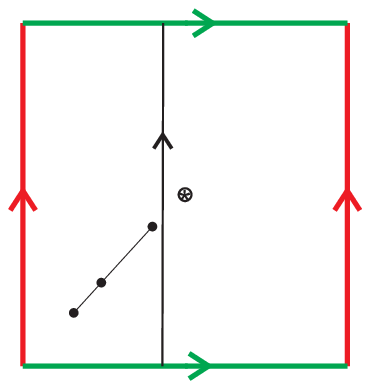}{1}}
  \newcommand\uybackdown{\pic{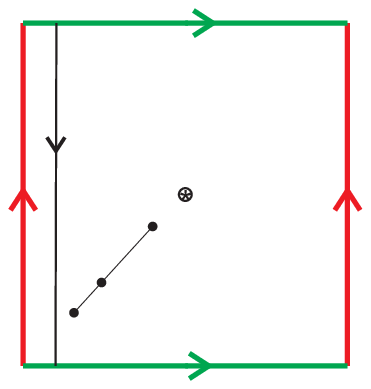}{1}}
 \newcommand\uxfrontleft{\pic{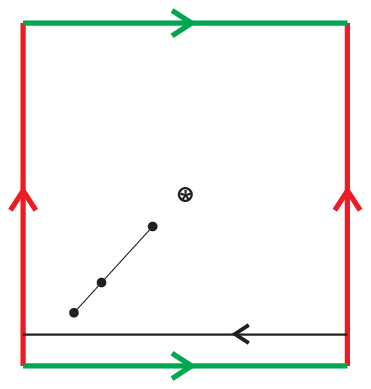}{1}}
 \newcommand\uxbackleft{\pic{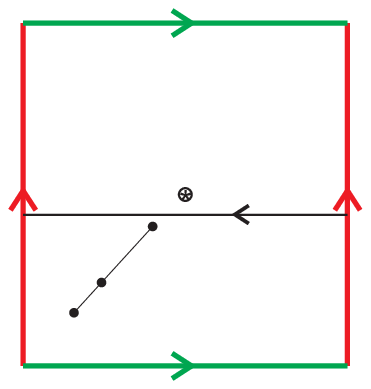}{1}}
 \newcommand\uyfrontdown{\pic{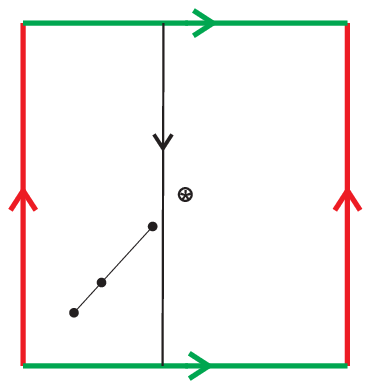}{1}}
  \newcommand\uyback{\pic{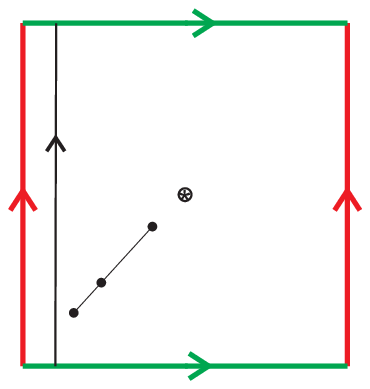}{1}}

 \newcommand\frontid{\pic{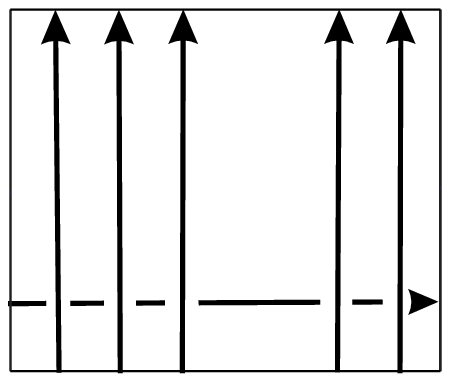}{.5}}
 \newcommand\backid{\pic{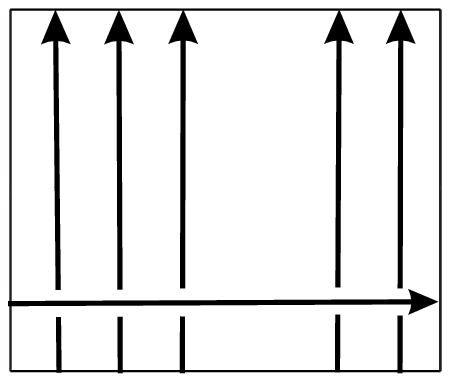}{.5}}

\newcommand\commutatoronebraid{\pic{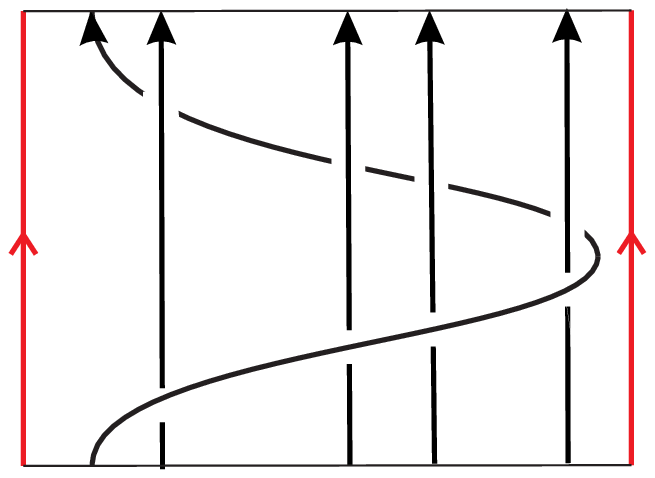}{.7}}
\newcommand\Pbraid{\pic{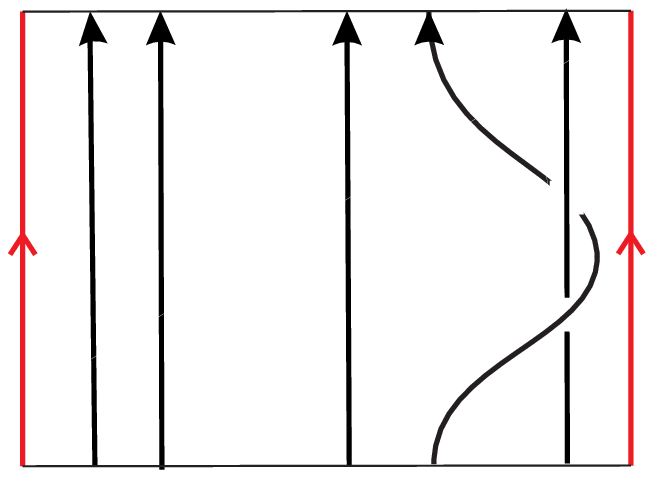}{.5}}
\newcommand\Pplan{\pic{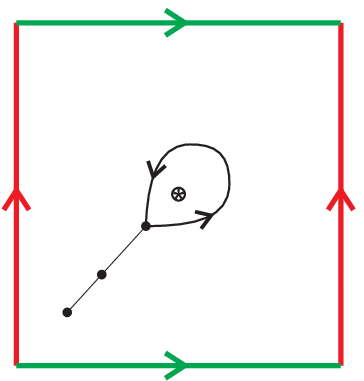}{.7}}
\newcommand\basecross{\pic{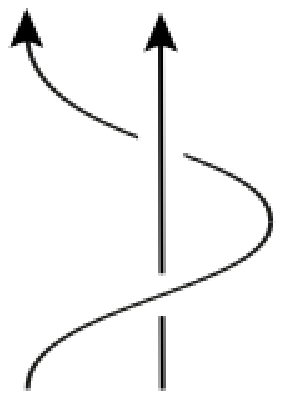}{.5}}
\newcommand\baseidentity{\pic{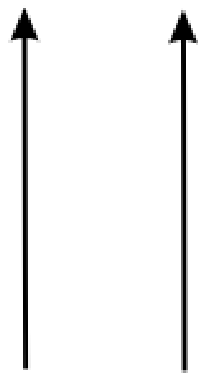}{.5}}

\newcommand\switchcurve{\pic{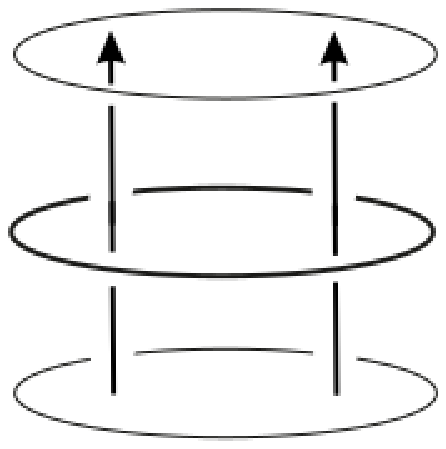}{.4}}
\newcommand\standardswitchcurve{\pic{switchcurve}{.5}}

\newcommand\switch{\pic{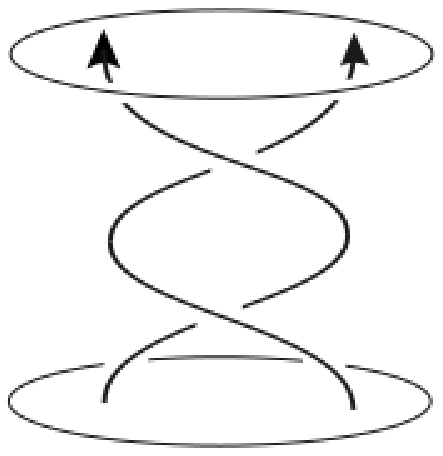}{.4}}
\newcommand\smooth{\pic{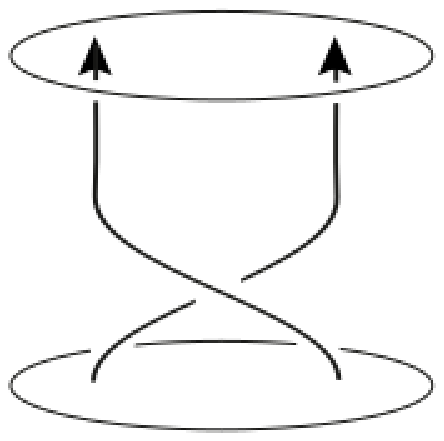}{.4}}

\newcommand\switchconfig{\pic{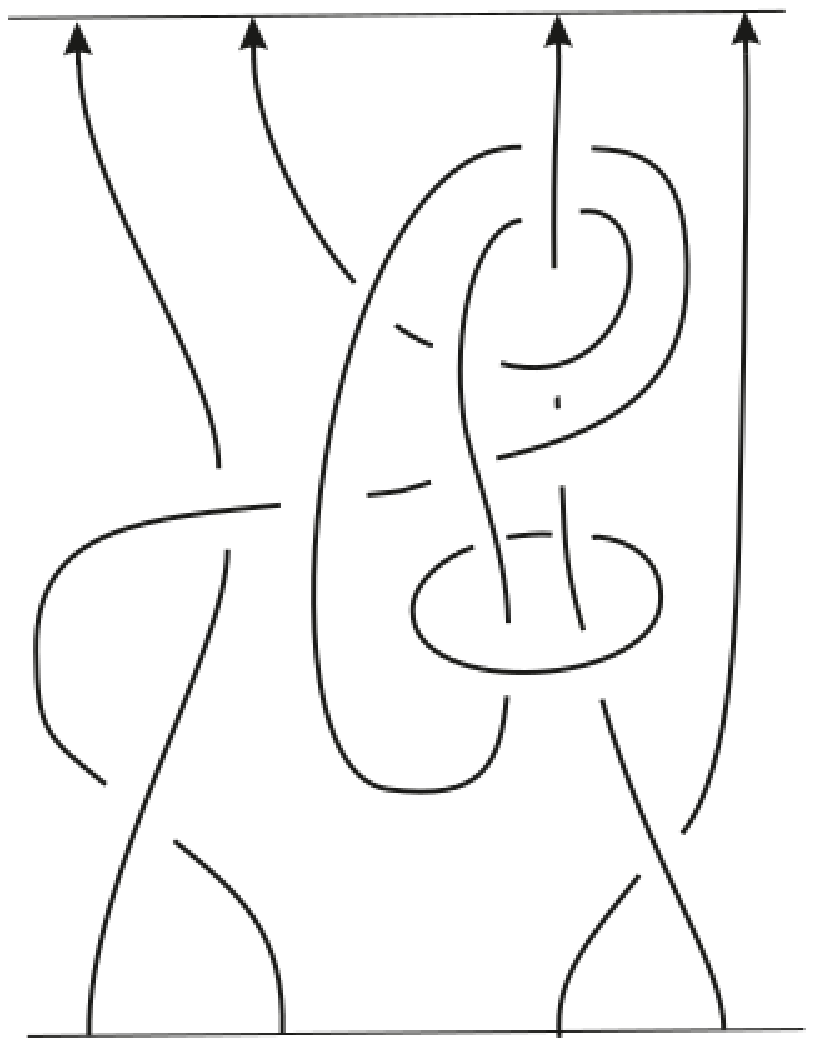}{.4}}
\newcommand\identityconfig{\pic{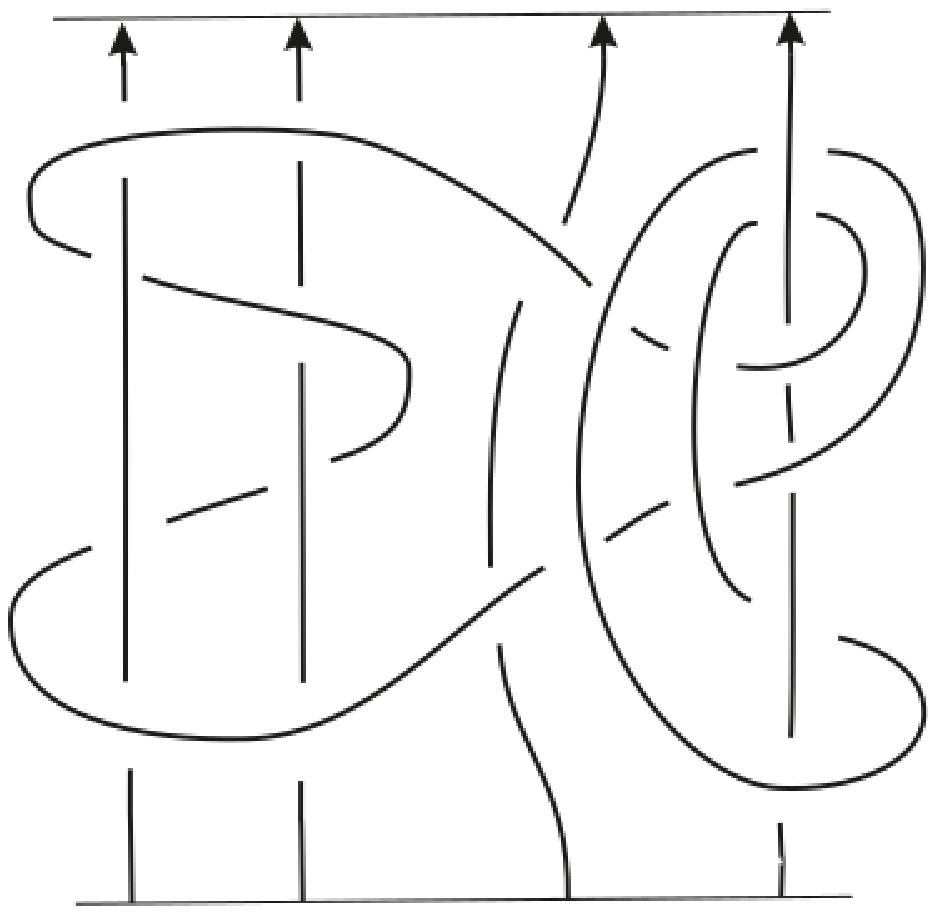}{.4}}
\newcommand\compositeconfig{\pic{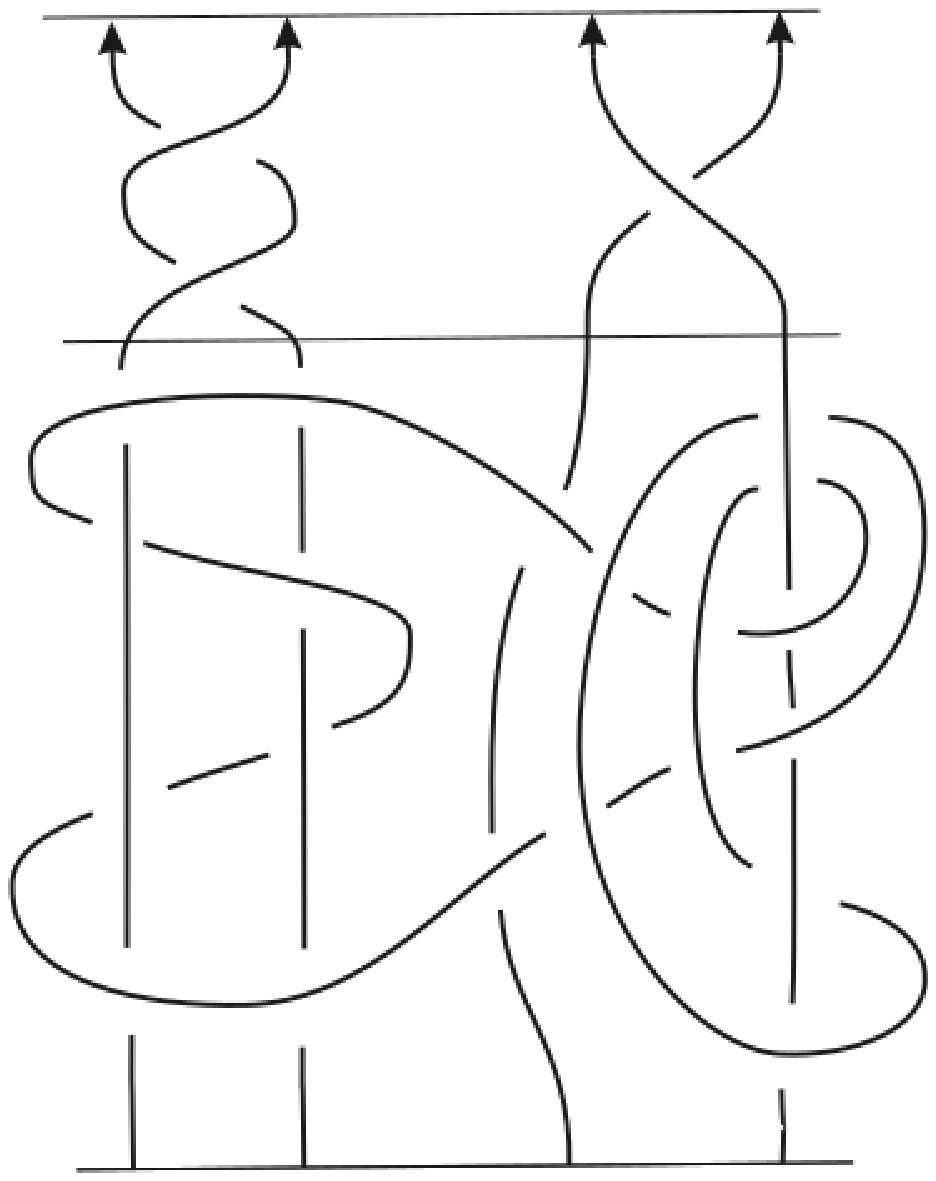}{.4}}
\newcommand\switchplan{\pic{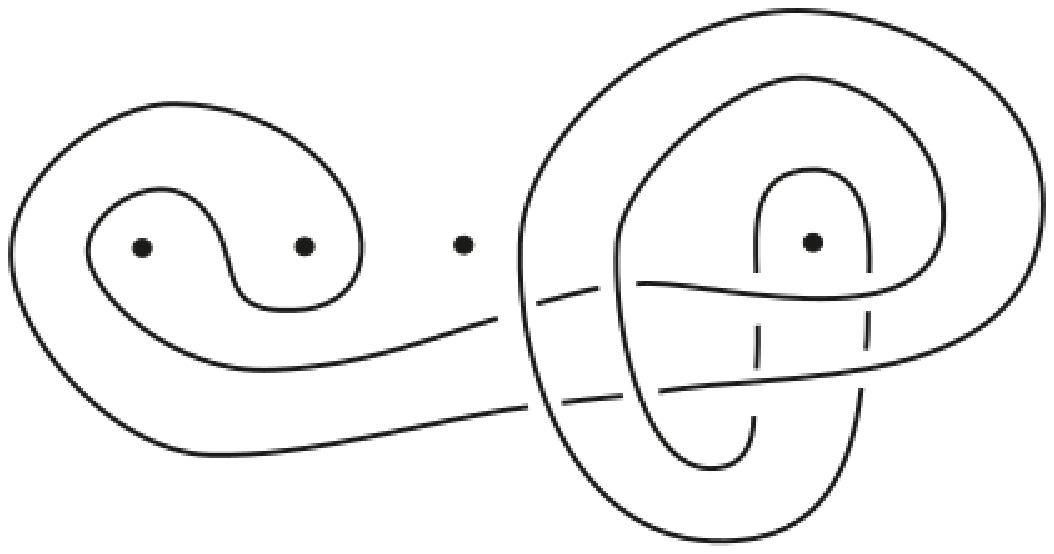}{.45}}
\newcommand\switcheddiagram{\pic{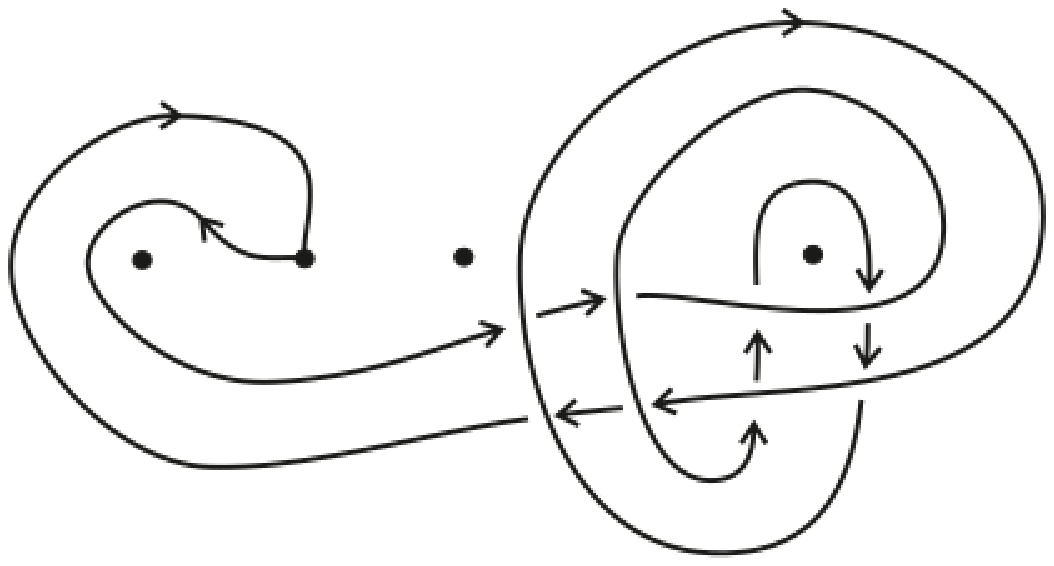}{.45}}
\newcommand\embeddeddiagram{\pic{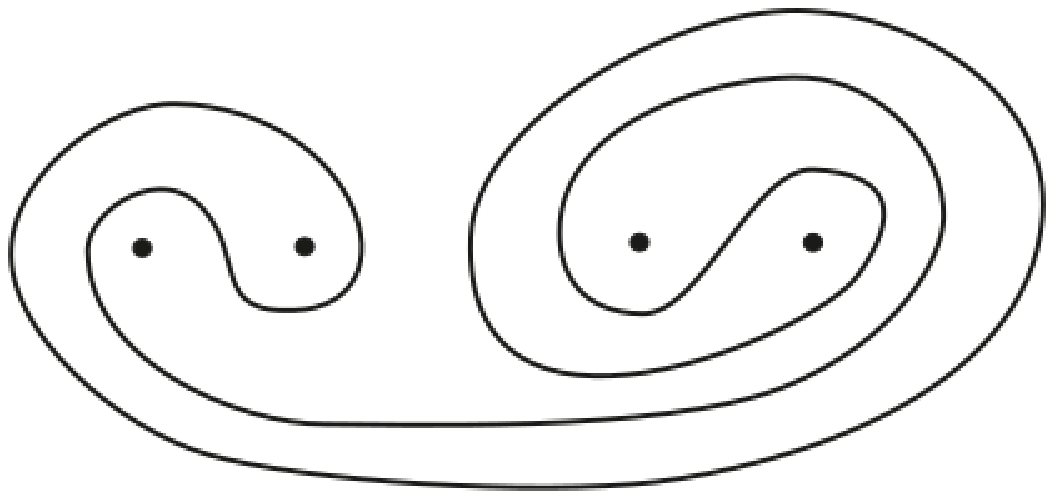}{.4}}

\newcommand\Xorband{\pic{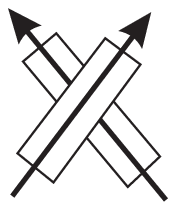}{.50}}
\newcommand\Yorband{\pic{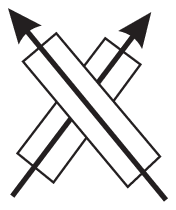} {.50}}
\newcommand\Iorband{\pic{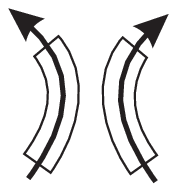} {.50}}

\def\Idorband{\pic{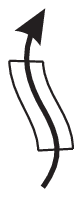} {.60}}
\def\Idorrightband{\pic{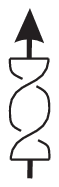} {.60}}
\def\Idorleftband{\pic{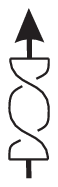} {.60}}

\def\rcurlorband{\pic{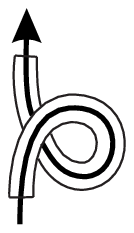} {.60}}
\def\lcurlorband{\pic{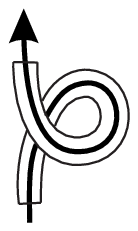} {.60}}

\newcommand {\sk}{\mathrm{Sk}}
\newcommand {\bsk}{\mathrm{BSk}}

\newcommand{\Q}{\mathbb{Q}}
\newcommand{\SL}{\mathrm{SL}}

\newcommand{\ZZ}{\mathbf{Z}}

\newcommand{\oursetminus}{-}

\begin{document}

\title{DAHAs and Skein theory
}

\author{H. R. Morton}
\email{su14@liverpool.ac.uk }
\address{Department of Mathematics, University of Liverpool, UK}
\author{Peter Samuelson}
\email{psamuels@ucr.edu}
\address{Department of Mathematics, University of California, Riverside, USA}
\maketitle

\begin{abstract} 
We give a skein-theoretic realization of the $\mathfrak{gl}_n$ double affine Hecke algebra of Cherednik using braids and tangles in the punctured torus. We use this to provide evidence of a relationship we conjecture between the skein algebra of closed links in the punctured torus and the elliptic Hall algebra of Burban and Schiffmann.
\end{abstract}

\tableofcontents

\section{Introduction}
This paper concerns a relation between the double affine Hecke algebras of Cherednik and certain algebras associated to the punctured and unpunctured torus that are defined using skein theory. As partial motivation, we first briefly discuss previous results in the $\mathfrak{sl}_2$ case of the Kauffman bracket, and then go on to discuss the conjectures and results of the present paper. We note that double affine Hecke algebras and skein algebras are related to several aspects of mathematical physics, including refined Chern-Simons theory \cite{AS15}; however, the exposition of the present paper will be purely mathematical.
\subsection{The Kauffman bracket}
The \emph{Kauffman bracket skein algebra} $K_s(F)$ of a surface $F$ is spanned by embedded links in the thickening $F \times [0,1]$, modulo the Kauffman bracket skein relations. These  are local relations depending on a parameter $s \in \mathbb C^\times$, and which are similar to equation \eqref{eq:skeinrel}. The product is given by stacking links in the $[0,1]$ direction, and this algebra can be viewed as a quantization of the ring of functions on the $\SL_2$ character variety of $F$ \cite{PS00, BFK99}. For the torus and punctured torus these algebras have been described explicitly by Frohman, Gelca, and by Bullock, Przytycki, respectively.

The \emph{double affine Hecke algebra} was defined by Cherednik (see, e.g.\ \cite{Che05} and references therein) using explicit generators and relations, and it depends on two parameters, $q,t \in \mathbb{C}^\times$.  In rank 1, its \emph{spherical subalgebra}
$SH_{q,t}$ 
of the DAHA was described explicitly by Koornwinder and later by Terwilliger. 

Combining these explicit descriptions leads to the following theorem.
\begin{theorem}[{\cite{FG00, Ter13, Koo08}}]\label{thm:ktorus}
	There is an isomorphism 
	\[
	K_s(T^2) \cong SH_{s,s}
	\] 
	between the Kauffman bracket skein algebra of the torus and the $t=q=s$ specialization\footnote{Technically, Frohman and Gelca showed skein algebra is isomorphic to the $t_{DAHA}=1$, $q_{DAHA}=s_{skein}$ specialization, but the presentations of Koornwinder and Terwilliger show that this spherical subalgebra is isomorphic to the spherical subalgebra in the $t_{DAHA} = q_{DAHA} = s_{skein}$ specialization, which is a nontrivial statement.} of the rank 1 spherical DAHA.
\end{theorem}


Combining the same algebraic theorems with the description of the skein algebra of the punctured torus instead, we obtain the following.
\begin{theorem}[{\cite{BP00, Ter13, Koo08}}]\label{thm:kptorus}
	There is a surjective map 
	\[
	K_s(T^2 \oursetminus D^2) \twoheadrightarrow SH_{q=s,t}
	\] 
	from the skein algebra of the punctured torus to the spherical rank 1 DAHA.
\end{theorem}
We note that the source algebra still only depends on one parameter -- the second parameter $t$ in the target appears in the relations describing the kernel of the map. One rough way of thinking of these results is that the spherical DAHA can be obtained from the skein algebra of the punctured torus using some kind of ``decoration'' at the puncture.

Let us also comment briefly on the importance of two parameters. The Macdonald polynomials are symmetric polynomials depending on the parameters $q$ and $t$ which have been studied intensively, and this had led (at the very least) to very interesting combinatorics, geometry, and algebra. In the $t=q$ specialization, the Macdonald polynomials degenerate to Schur polynomials, which are much more well-understood. It is therefore desirable to be able to ``see'' both parameters from topology.

\subsection{The elliptic Hall algebra and Homflypt skeins}
This paper deals with the ``infinite rank'' versions of the algebras in the previous section. The Kauffman bracket skein algebras are replaced with the \emph{Homflypt skein algebra} $\sk(T^2)$ of closed links in the thickened torus modulo the Homflypt skein relations. These relations are recalled in equations \eqref{eq:skeinrelintroc} and \eqref{eq:skeinrelintrof}, and they depend on parameters $s,v \in \mathbb{C}^\times$. The spherical DAHA is replaced by the \emph{elliptic Hall algebra} $\E_{\sigma, \bar \sigma}$ defined by Burban and Schiffmann \cite{BS12}. In earlier work we proved the analogue of Theorem \ref{thm:ktorus}:

\begin{theorem*}[{\cite{MS17}}]
	There is an isomorphism 
	\[
	\sk(T^2) \cong \E_{s,s}
	\]
	between the Homflypt skein algebra of the torus and the $\sigma=\bar \sigma = s$ specialization\footnote{To be precise, the presentation of $\sk(T^2)$ does not depend on the parameter $v$, so technically the right hand side of the isomorphism should be $\E_{s,s}\otimes_k \mathbb{C}[v^{\pm 1}]$. Also, see Remark \ref{rmk:params} for a comparison of this specialization to the $q=1$ specialization.} of the elliptic Hall algebra.
\end{theorem*}

We make the following conjecture which is the analogue of Theorem \ref{thm:kptorus}. Note that as in the Kauffman bracket case, the source algebra only depends on one parameter $s$ -- we expect the second parameter in the elliptic Hall algebra to arise as in the kernel of the map from $\sk(T^2 \oursetminus D^2)$. Also, we point out that $\sk(T^2 \oursetminus D^2)$ is ``much bigger'' than the elliptic Hall algebra: as a vector space it is isomorphic to a polynomial algebra with generators ``conjugacy classes in the free group of rank 2,'' while the elliptic Hall algebra is isomorphic as a vector space to a polynomial algebra with generators indexed by $\mathbb{Z}^2$.
\begin{conjecture}\label{conj:punctoeha}
	There is a surjective algebra map $\sk(T^2\oursetminus D^2) \twoheadrightarrow \E_{\sigma, \bar \sigma}$. This map takes a simple closed curve of homology class $\xx \in \Z^2$ to the generator $u_\xx \in \E_{\sigma, \bar \sigma}$ used by Schiffmann and Vasserot.
\end{conjecture}
The currently available proofs of all three of the previous theorems involve giving explicit presentations of the algebras in the statements, and then using these presentation to construct an algebra map by hand. Giving a presentation of $\sk(T^2\oursetminus D^2)$ seems difficult, so instead of doing this, we give some evidence for this conjecture using other techniques, which we describe in the next subsection. These techniques are closely related to some techniques used or mentioned by others -- one reason we make precise statements of our own version is that it gives evidence for the conjecture above.

\subsection{DAHAs for $\gl_n$}


The elliptic Hall algebra $\E_{\sigma, \bar \sigma}$  is closely related to the double affine Hecke algebras $\daha_n$ of type $\gl_n$, as detailed in the work of Schiffman and Vasserot, \cite{SV11}. We were intrigued by the nature of the presentation of the  algebras $\daha_n$, which involved Homfly type relations and braids in the torus $T^2$. This led us to speculate on the possibility of constructing some form of skein theoretic model which would incorporate both the algebra $\daha_n$, in terms of braids,  and our original algebra of closed curves in the thickened torus. 

As a start we considered the possibility of a direct skein-based model in terms of $n$-braids in $T^2$ for the double affine Hecke algebra $\daha_n$ with parameters $t$ and $q$. 

The naive approach of considering $n$-braids in $T^2$ modulo the Homfly relations gives a model that works for one of the parameters $t$ but only covers the case $q=1$ for the other parameter.  A search of the literature came up with a paper by Burella et al, \cite{BWPV14}, suggesting that a model based on framed braids could handle the more  general case of $q\ne 1$, where adding a twist to the framing of a braid was reflected in multiplication by $q$. Their model depends on the product of certain braids with explicit framing resulting in a single twist on the framing of one string. We tried without success to follow the diagrammatic views of this product, which appears to us to have the trivial framing on all strings, and not the desired twist. We  worked out a uniform way of specifying a framing on the strings of a torus braid, noted in Theorem \ref{thm:dahatoskein} below, and we came to the conclusion that the use of framing alone would not provide a means of incorporating the second parameter $q$ into a geometric model for $\daha_n$.

We were still hopeful of making a skein-based geometric model, and we came up instead with one that includes an extra string. Instead of working with $n$-braids in the torus we use $(n+1)$-braids in which one distinguished string, called the \emph{base string}, is fixed throughout. Equivalently our geometric elements are $n$-braids in the once-punctured torus, regarding the fixed base string as determining the puncture.  In our model we use linear combinations of these braids. The regular $n$-string braids are allowed to interact as braids by the Homfly relations  and the parameter $q=c^{-2}$ is introduced when a regular string is allowed to cross through the base string. 

These relations can be summarised in diagrammatic form as \[\Xor -\Yor=(s-s^{-1})\Ior \]

\bc
\labellist\small
\pinlabel{$*$} at 146 740
\endlabellist\Xor$\quad=\quad c^2\ $
\labellist\small
\pinlabel{$*$} at 146 740
\endlabellist\Yor
\ec


In Section \ref{braidskein} of this paper we set up our skein model starting from ${\bf Z}[s^{\pm1},c^{\pm1}]$-linear combinations of $n$-braids in the punctured torus, up to equivalence.   We give a presentation for this algebra as a quotient of the group algebra of the braid group of $n$-braids in the punctured torus, using an explicit presentation by Bellingeri, \cite[Theorem 1.1]{B04}, for this braid group with generators $\sigma_1,\ldots,\sigma_{n-1}, a, b$. Our emphasis here is on the use of geometric diagrams to represent the elements of the algebra. Such an approach is used elsewhere with oriented framed (banded) curves in a variety of manifolds as the basic ingredients subject to the $3$-term linear relations above. A new addition in our current setting is the use of the base string, and the relation introducing the second parameter $c^2$ when a string is moved across it.


We give diagrammatic illustrations of some useful braids and their interrelations, and show how to interpret Bellingeri's presentation in terms of our braids.

The skein relations can then be included by adding the relations
\[\sigma_i-\sigma_{i}^{-1}=s-s^{-1}\] or \[(\sigma_i-s)(\sigma_i+s^{-1})=0\] in  quadratic form, and \[P=c^2,\] where $P$ is the braid taking string $n$ once round the puncture and fixing the other strings. There is a simple formula for $P$ in terms of the generators of the punctured braid group, which we provide.

The outcome is a presentation for the algebra of braids in the punctured torus modulo the skein relations, which is our skein-based model, $\bsk_n(T^2,*)$, see Definition \ref{def:bskein}.
We establish this presentation in Theorem \ref{braidpresentation}, and show how it corresponds exactly to the presentation in \cite{SV11} for the double affine Hecke algebra $\daha_n$:
\begin{theorem*}[see Thm.\ \ref{braidpresentation}]
	The braid skein algebra $\bsk_n(T^2,*)$ is isomorphic to the $\mathfrak{gl}_n$ double affine Hecke algebra $\ddot H_{n;q,t}$.
\end{theorem*}
\begin{remark}
While this paper was in preparation, D. Jordan and M. Vazirani proved a very similar statement, that the DAHA is a quotient of the group algebra of the braid group of the punctured torus \cite[Prop. 4.1]{JV17}. One difference between their statement and ours is that they impose relations algebraically, while we impose ours topologically, giving a visually appealing interpretation of elements of the algebra and the underlying relations. More precisely, their relations are a subset of ours; it is therefore a-priori possible that we impose strictly more relations than they do, and the content of this theorem is that their relations imply ours. Let us also mention here that Cherednik states in \cite{Che05} that the $\mathfrak{gl}_n$ DAHA is a deformation of the Hecke quotient of the braid group of the \emph{closed} torus. One of our desires was to remove the word ``deformation'' from this statement so that the deformation parameter $q$ has a direct topological meaning.
\end{remark}


In  Section \ref{fullskein} we make use of framed $n$-tangles in the full framed Homfly skein of the punctured torus, $\sk_n(T^2,*)$. In this setting an $n$-tangle consists of $n$ framed oriented arcs in the thickened torus, along with a number of framed oriented closed curves. The arcs are no longer restricted to lying as braids in $T^2\times I$. We work with linear combinations of framed tangles and  impose the local relation 
\begin{equation}\label{eq:skeinrelintroc}
\Xorband\  -\ \Yorband \qquad =\qquad{(s-s^{-1})}\quad\ \Iorband
\end{equation}
between framed tangles, as well as the  change of framing relation 
\begin{equation}\label{eq:skeinrelintrof}
\rcurlorband \qquad=\qquad {v^{-1}}\quad \Idorband
\end{equation}
using a second parameter $v$. In keeping with the first section we include a base string defining the puncture, and allow a framed string to cross through it at the expense of the scalar $c^2$, with local relation
\bc
\labellist\small
\pinlabel{$*$} at 146 740
\endlabellist\Xor$\quad=\quad c^2\ $
\labellist\small
\pinlabel{$*$} at 146 740
\endlabellist\Yor .
\ec

There is a homomorphism from $\bsk_n(T^2,*)$ to $\sk_n(T^2,*)$, since the braids in $\bsk_n(T^2,*)$ can be given a  consistent framing using a nonvanishing vector field on the torus so that the relations in $\bsk_n(T^2,*)$ continue to hold in the wider tangle skein.
It is not clear however whether this homomorphism is injective. One point at issue is that, as well as the extra elements introduced, the  additional relations between them might have the effect of collapsing the algebra considerably. Nonetheless, we make the following:

\begin{conjecture}\label{conj:mainiso}
	The map $\bsk_n(T^2,*) \to \sk_n(T^2,*)$ from the braid skein algebra to the tangle skein algebra is an isomorphism. 
\end{conjecture}
\begin{remark}
This may seem surprising, since tangles can contain closed curves, which means that a-priori, the tangle algebra is ``much bigger'' and the map in question should not be surjective. Indeed, for other surfaces the analogous map is not surjective. However, on the torus, an embedded closed curve ``on one side of the puncture'' is isotopic to the same curve ``on the other side of the puncture," and we use this fact in Theorem \ref{thm:dahatoskein} to show that the map in the conjecture is surjective. 
\end{remark}



There is an algebra map from the (classical) Homflypt skein module $\sk(T^2\oursetminus D^2)$ of closed links in the thickened punctured torus to our skein algebra $\sk_n(T^2, *)$ of tangles with a base string, given by ``filling in the puncture with the identity braid and the base string.'' If we assume Conjecture \ref{conj:mainiso}, this gives us an algebra map 
$\sk(T^2 \oursetminus D^2) \to \ddot H_n$ for any $n$. We can compose this map with multiplication by the symmetrizer $\e$ in the finite Hecke algebra to obtain a map $\sk(T^2\oursetminus D^2) \to \e \ddot H_n \e$ to the so-called \emph{spherical subalgebra}.

Schiffmann and Vasserot showed that the elliptic Hall algebra is the $n \to \infty$ limit of the spherical subalgebras (see Theorem \ref{thm:SVlimit} for a precise statement). Let $\sk^+(T^2 \oursetminus D^2)$ be the subalgebra generated by ``curves lifted from the closed torus which only cross the $y$-axis positively'' (see Definition \ref{def:posskein} for a precise statement). We then show the following, which we view as evidence for Conjecture \ref{conj:punctoeha}.

\begin{theorem*}[see Thm. \ref{thm:psktoeha}]
	Assuming Conjecture \ref{conj:mainiso} holds, there is a surjective algebra map $\sk^+(T^2\oursetminus D^2) \twoheadrightarrow \E^+_{\sigma, \bar \sigma}$. This map sends the simple closed curve of homology class $\xx \in \Z^2$ to the generator $u_\xx$ of the elliptic Hall algebra.
\end{theorem*}

One corollary of this theorem (again assuming Conjecture \ref{conj:mainiso}) is that the generator $u_\xx$ of the elliptic Hall algebra has a simple interpretation as a sum $W_\xx$ of certain closed curves on the torus, with homology class $\xx \in \Z^2 = H_1(T^2 \oursetminus D^2)$. In fact, these elements are lifts of the exact same elements in $\sk(T^2)$ that were used in \cite{MS17}. The subtle point here is that not all the relations between $W_\xx$ that were proved in \cite{MS17}  hold in the punctured torus, since the proofs of some of these relations used global isotopies on $T^2$ that don't lift to the punctured torus. Roughly, the problem is that some curves ``get caught on the puncture.'' When pushed into our skein algebra of tangles with a base string, these curves can once again be pushed through the puncture, but at the cost of some ``lower order terms'' involving braids, and these lower order terms contribute to the generating series relations in the elliptic Hall algebra. See also Remark \ref{rmk:ptorusrel}.

This suggests one purely algebraic question of possible interest. The Schiffmann-Vasserot elements $u_\xx$ are in the spherical DAHA $e_n \daha_{q,t} e_n$, where $e_n$ is the symmetrizer in the finite Hecke algebra. However, the images of our elements $W_\xx$ most naturally lie in the centralizer $Z_{\ddot H_n}(H_n)$ of the finite Hecke algebra $H_n$ inside the double affine Hecke algebra. This suggests there may be an interesting limit of these centralizers which would include the elliptic Hall algebra a subalgebra.

Let us briefly comment on related or future work. In \cite{JV17}, Jordan and Vazarani used factorization homology to construct representations of the braid-skein algebra $\bsk_n(T^2,*)$, and more skein-theoretic techniques to construct representations are being used in work in progress of Vazarani and Walker. We hope that some combination of these approaches could be used to prove Conjecture \ref{conj:mainiso}, but we don't discuss this in the present paper. 

We also note that the so-called $A_{q,t}$ algebra introduced by  Carlesson and Mellit in \cite{CM18} has a relation that looks like a 3-term version of the skein relation involving the base string. Discussions with Jordan and Mellit indicate that more precise versions of this statement are available, but this will be left to future work.

A summary of the contents of the paper is as follows. In Section 2 we recall algebraic background involving DAHAs and the elliptic Hall algebra. In Section 3 we define the braid skein algebra and show it is isomorphic to the DAHA, and in Section 4 we discuss the tangle skein algebra. In Section 5 we compare the tangle skein algebra and the (classical) skein algebra of closed links in the punctured torus to the elliptic Hall algebra.

\noindent \textbf{Acknowledgements:} This work was initiated during the authors participation in the Research in Pairs program at Oberwolfach in the spring of 2015, and we gratefully acknowledge their support for our stay there, and for their excellent working conditions. More work was done at conferences at the Isaac Newton Institute and at BIRS in Banff, and we gratefully acknowledge their support. Parts of the travel of the second author were supported by a Simons Travel Grant. We thank E.\ Gorsky, A.\ Negut, A.\ Oblomkov, O.\ Schiffmann,  E.\ Vasserot, M.\ Vazirani, and K. Walker for their interest and discussions of this and/or their work over the years. We especially thank D. Jordan and A. Mellit for many discussions closely related to this paper. We would also like to thank the referees for thoughtful comments and suggestions that helped us improve the exposition and clarity of the paper. We are grateful to M. Scharlemann for discussions in connection with our revision of Section  \ref{sec:isorels}. 

The work of the second author has been partially funded by the ERC grant 637618 and a Simons Foundation Collaboration Grant.


\section{Algebraic background}\label{sec:algebra}
In this section we recall the algebraic definitions and results that we need in the rest of the paper. In particular, we define the elliptic Hall algebra and double affine Hecke algebras (DAHAs), and we recall results of Schiffmann and Vasserot relating the two. In later sections we use their results to relate the skein algebra of the punctured torus to the elliptic Hall algebra.

\subsection{The Elliptic Hall algebra}
Let us recall the definition of the elliptic Hall algebra $\E = \E_{\sigma, \bar \sigma}$ of Burban and Schiffmann \cite{BS12}, using the conventions of \cite{SV11}. It is an algebra over the ring $\Q(\sigma, \bar \sigma)$, and it is generated by elements $u_\xx$ for $\xx \in \Z^2$, subject to the following relations:
\begin{enumerate}
	\item If $\xx$ and $\xx'$ belong to the same line in $\Z^2$, then $[u_\xx,u_{\xx'}] = 0$.
	\item Assume that $\xx$ is primitive and that the triangle with vertices $0$, $\xx$, and $\xx+\yy$ has no interior lattice points. Then
	\begin{equation}\label{eq:hallrel}
	[u_\yy, u_\xx] = \epsilon_{\xx,\yy} \frac {\theta_{\xx+\yy}}{\alpha_1}
	\end{equation}
	where the elements $\theta_\zz$ with $\zz \in \Z^2$ are obtained by the generating series identity
	\[
	\sum_{i} \theta_{i \xx_0} z^i = \exp\left( \sum_{i \geq 1} \alpha_i u_{i \xx_0} z^i\right)
	\]
	for $\xx_0 \in \Z^2$ primitive.
\end{enumerate}
In the above relations we used the constants $\epsilon_{\xx,\yy} = \mathrm{sign}(\det(\xx\,\yy))$ and 
\[
\alpha_i = (1-\sigma^i)(1-\bar \sigma^i)(1-(\sigma \bar \sigma)^{-i})/i
\]
We also define the following subsets of $\ZZ := \Z^2$:
\begin{equation}
\ZZ^> := \{(x,y) \mid x > 0\},\quad \ZZ^+ :=  \ZZ^> \sqcup \{(0,y) \mid y \geq 0\}
\end{equation}
We also use this notation to define subalgebras of $\E$, for example, 
\[
\E^+ := \langle u_\xx \mid \xx \in \ZZ^+\}
\]
We will use similar notation for other algebras generated by elements indexed by $\ZZ$.
Finally, let $d(\xx)$ be the greatest common denominator of the entries of $\xx \in \Z^2$.

\subsection{Limits of DAHAs}
We now recall the definition of the double affine Hecke algebra $\daha_n$, following the conventions given in \cite{SV11}.
This is an algebra over ${\bf Z}[t^{\pm1/2}, q^{\pm1}]$ with generators \[\{T_i\},1\le i\le n-1, \quad \{X_j\},\{Y_j\}, 1\le j\le n\] and relations
\begin{eqnarray}
(T_i+t^{1/2})(T_i-t^{-1/2})&=&0\\
T_i T_{i+1} T_i&=& T_{i+1}T_i T_{i+1}\\[0mm]
[T_i,T_j]&=&0, |i-j|>1\\[0mm]
[T_i,X_j]=[T_i,Y_j]&=&0, j\ne i,i+1\\[0mm]
[X_i,X_j]=[Y_i,Y_j]&=&0\\
X_{i+1}&=&T_iX_iT_i, \\
Y_{i+1}&=&T_i^{-1}Y_i T_i^{-1}\\
X_1^{-1}Y_2&=&Y_2X_1^{-1}T_1^{-2}\\
Y_1 X_1\cdots X_n&=&q X_1\cdots X_n Y_1
\end{eqnarray}

Let $e_n$ be the symmetrizing idempotent in the finite Hecke algebra (which is generated by the $T_i$'s), which is characterized by $T_je_n = e_n T_j = t^{1/2}e_n$ for all $j$. The spherical DAHA is the subalgebra $\SH^n_{q,t} := e_n \daha^n_{q,t} e_n$ of $\daha^n_{q,t}$, and it is also $\Z^2$-graded. There is an  $\SL_2(\Z)$ action on the subalgebra $\SH_{q,t}^n$ (see the paragraph above Lemma 2.1 in \cite{SV11}).



Following \cite[Sec.\ 2.2]{SV11} (except for the notational change $P\to Q$), for $k > 0$ we define elements 
\begin{equation*}
Q^n_{0,k} = e_n \sum_i Y_i^k e_n
\end{equation*}
Elements $Q^n_\xx$ for $\xx \in 
\Z^2$ are defined using the $\SL_2(\Z)$ action. We define $\SH_{q,t}^{n,>}$ to be the subalgebra of $\SH_{q,t}^n$ generated by $Q_{a,b}^n$ with $a > 0$.


Let us identify parameters $\sigma = q^{-1}$ and $\bar \sigma = t^{-1}$. Then Schiffmann and Vasserot proved the following theorem relating the elliptic Hall algebra and spherical DAHAs.
\begin{theorem}[{\cite[Thm.\ 3.1]{SV11}}]\label{thm:sv}
	The assignment 
	\[
	u_\xx \mapsto \frac 1 {q^{d(\xx)} - 1}Q_\xx^n
	\]
	extends uniquely to a $\Z^2$-graded $\SL_2(\Z)$-equivariant surjective algebra homomorphism 
	\[
	\phi^n: \E_{q,t} \twoheadrightarrow \SH_{q,t}^{n}
	\] 
\end{theorem}

Given the previous theorem, a natural question is whether there is some type of limit one can take as $n \to \infty$. It turns out that there is, but to describe it Schiffmann and Vasserot first had to prove the following theorem.

\begin{theorem}[{\cite[Prop.\ 4.1]{SV13}}]
	The assignment $Q^n_\xx \mapsto Q^{n-1}_\xx$ for each $\xx \in \ZZ^+$ extends to a unique surjective algebra map $\Phi_n: \SH_{q,t}^{n,+} \to \SH_{q,t}^{n-1,+}$.
\end{theorem}

This theorem allows us to construct a projective limit $\varprojlim \SH^n_{q,t}$. Also, the generators $Q^n_\xx$ provide elements in this projective limit, and we let $\SH^{\infty,+}_{q,t}$ be the subalgebra generated by these elements for $\xx \in \ZZ^+$. Theorem \ref{thm:sv} shows that there is a map from the elliptic Hall algebra to $\SH_{q,t}^{\infty,+}$.


\begin{theorem}[{\cite[Thm.\ 4.6]{SV13}}]\label{thm:SVlimit}
	The induced map $\phi^\infty: \E_{q,t}^+ \to \SH_{q,t}^{\infty,+}$ is an isomorphism.
\end{theorem}

Summarizing this work of Schiffmann and Vasserot, we obtain the following corollary which we use below.

\begin{corollary}\label{cor:limitmap}
	Suppose $A$ is an algebra generated by elements $a_\xx$ for $\xx \in S \subset \ZZ^+$. Suppose there are algebra maps $A \to \SH_{q,t}^{n,+}$ for each $n$ such that $a_\xx \mapsto Q_\xx$. Then there is an algebra map $A \to \E_{q,t}^+$ sending $a_\xx \mapsto (q^{d(\xx)}-1) u_\xx$.
\end{corollary}

\begin{remark}\label{rmk:params}
	The elliptic Hall algebra could instead be defined using three constants $q_1, q_2, q_3$ which replace the symbols $\sigma, \bar \sigma, (\sigma \bar \sigma)^{-1}$ (so that the definition above could be recovered by specializing $q_1=\sigma$, $q_3 = \bar \sigma$, and $q_3 = q_1^{-1}q_2^{-1}$).  These three parameters $q_i$ then appear symmetrically in the presentation, and can therefore be permuted to give an isomorphic algebra. To the best of our knowledge, this symmetry seems mysterious from every point of view that appears in this paper. For example, in terms of elliptic curves over the finite field of order $q$, there is an identity $\sigma \bar \sigma = \sqrt{q}$, and this identity breaks the symmetry mentioned above. In terms of double affine Hecke algebras, the parameters $q$ and $t$ appear in the relations in very different ways, and there is again no symmetry in the parameters. Finally, at the topological level this symmetry does not seem to be visible, since the parameter $s$ appears as a relation between two strands in a tangle, while the parameter $c$ appears as a relation involving the base string and a strand in a tangle. 
	
	Let us also note that the skein relation we impose with the base string makes the base string ``invisible'' when $q=1$. This means that the $q=1$ specialization of the conjectures in the present paper should be compatible with the results of the previous paper \cite{MS17}; however, the previous paper states that the skein algebra of the torus is isomorphic to the $q=t$ specialization of the elliptic Hall algebra. 
	This apparent conflict is explained by the symmetry of parameters in $\E_{\sigma, \bar \sigma}$ described in the previous paragraph. In particular, this symmetry implies that the three specializations $\E_{q,1}$, $\E_{1,q}$, and $\E_{q,q}$ are all isomorphic. (The reason the $q=t$ specialization was chosen in \cite{MS17} is that this specialization is the one which is compatible with the actions of both algebras on the space of symmetric functions, since it is the specialization taking Macdonald functions to Schur functions.)
\end{remark}

\section{Skeins with a base string}\label{braidskein}

We will describe some skeins which use the framed Homfly relations on oriented  framed curves and braids in the thickened torus $T^2\times I$, together with a single fixed base string $\{*\}\times I\subset T^2\times I$.

In this section we define the braid skein algebra $\bsk_n(T^2,*)$ in terms of ${\bf Z}[s^{\pm1},c^{\pm1}]$-linear combinations of braids, and their composites, and prove that
it is isomorphic to the double affine Hecke algebra $\daha_n$ (following the conventions in \cite{SV11}). (See Theorem \ref{thm:bskiso}.) The multiplication in the braid skein algebra comes from stacking in the $[0,1]$ direction -- more precisely, it comes from the glue-then-rescale map $[0,1] \sqcup [1,2] \to [0,2] \to [0,1]$.

\subsection{Isotopies of braids in the punctured torus}

We start by considering the group of $n$-braids in the punctured torus $T^2\oursetminus \{*\}$. We will  work with the thickened torus $T^2\x I$ with a single fixed base  string $\{*\}\x I$ to determine by the  puncture $*\in T^2$. 
Braids are made up of $n$ strings oriented monotonically from $T^2\x \{0\}$ to $T^2\x \{1\}$ which do not intersect each other or the base string.
Braids are considered equivalent when the strings are isotopic avoiding the base string.

Composition of braids is defined by placing one on top of the other, using the convention that $AB$ means braid $A$ lying below braid $B$.  

As in \cite{MS17} we shall regard $T^2$ as given by identifying opposite pairs of sides in the unit square $[0,1]\times [0,1]$.
Take the base point $*$ to be the centre $(1/2,1/2)$ of the square. Fix $n>0$ points in order on the lower part of the diagonal of the square between $(0,0)$ and $*$ as the end points for $n$-string braids in $T^2\times I - \{*\}\times I$.

We can draw the thickened torus in plan view as a square with opposite pairs of edges identified. We  show the  braid points and the base string position  in the figure below, including a line along the diagonal through them as a visual help to keep track of them.
\bc
\Torusbase
\ec

We can indicate some simple braids where only one or two of the points move  by drawing the path of the moving points on the plan view, rather as in the diagrams in \cite{AM98}.  In this view the braid product is given by concatenation of the paths.

For example, write $ x_i$ for the braid in which point $i$ moves uniformly around the $(1,0)$ curve in the torus, and $ y_i$ where point $i$ moves around the $(0,1)$ curve, with all other points remaining fixed. These are shown in plan view as \bc $ x_i = $\xii,$\quad   y_i =$\etai
\ec
  and in a side view in figures \ref{xielevation} and \ref{yielevation}, where the colouring of the edges being identified is consistent with that used in the plan view. Similarly the braid $\sigma_i$ appears in plan view as in figure \ref{plansigmai},
\begin{figure}[ht]
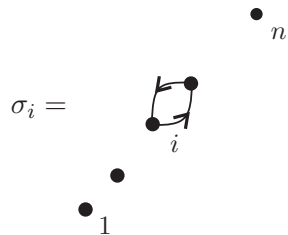

\bc $\sigma_i = $ \labellist\small
\pinlabel {$1$} at 100 337
\pinlabel {$i$} at 122 363
\pinlabel {$n$} at 155 398
\endlabellist \sigmaiplan
\ec
\caption{Plan view of $\sigma_i$}\label{plansigmai}
\end{figure}
 concentrating only on the region around the braid points.

A side elevation for $ x_i$ viewed in the $(0,1)$ direction is shown in figure \ref{xielevation}, 
\begin{figure}[ht]
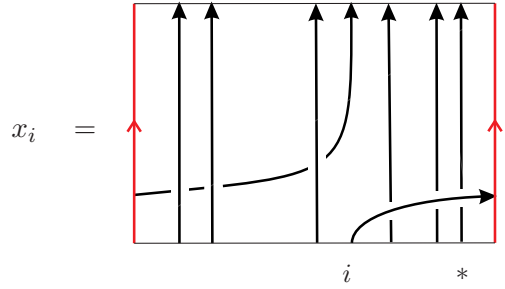

\bc $ x_i\quad = \quad $\labellist\small
\pinlabel {$i$} at 210 335
\pinlabel{$*$} at 273 335
\endlabellist \xiibraid
\ec
\caption{Side view of $x_i$}\label{xielevation}
\end{figure}
and $ y_i$ viewed in the $(-1,0)$ direction is seen in elevation in figure \ref{yielevation}.
 \begin{figure}[ht]
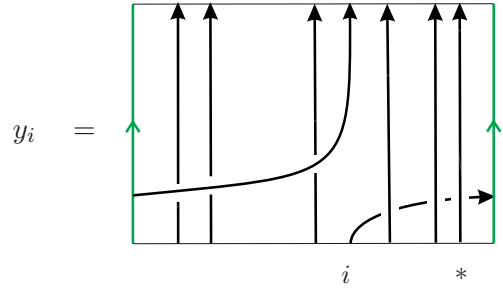

 \bc $ y_i\quad = \quad $\labellist\small
\pinlabel {$i$} at 210 335
\pinlabel{$*$} at 273 335
\endlabellist \etaibraid
\ec
\caption{Side view of $y_i$}\label{yielevation}
\end{figure}

Using either of these two elevation views the braids $\sigma_i$ appear in their usual form above, and it is immediate from these views that \begin{eqnarray}\sigma_i^{-1} x_i\sigma_i^{-1}&=& x_{i+1}\\
\sigma_i y_i\sigma_i&=& y_{i+1}.\end{eqnarray}

In a plan view we assume that paths are projections of  braid strings which rise monotonically from their initial braid point to their final braid point.  The product of two braids corresponds to the concatenation of their paths. 

We can see that the braids $\{ x_i\}$ commute among themselves, since their paths in the plan view are disjoint.  The same applies to the braids $\{ y_i\}$, and equally the braids $\sigma_i$  commute with  $ x_j$ and $ y_j$ when $j\ne i,i+1$.

The relations \[x_1x_2=x_2 x_1,\quad  y_1 y_2=y_2 y_1\] become
\begin{eqnarray}
 x_1\sigma_1^{-1}x_1\sigma_1^{-1} &=& \sigma_1^{-1}x_1\sigma_1^{-1} x_1,\\ 
 y_1\sigma_1 y_1\sigma_1&=& \sigma_1 y_1\sigma_1 y_1
\end{eqnarray}
in terms of the generators $x_1, y_1$.

We can use the plan view for a braid where two paths cross, taking the usual convention of knot crossings to show which strand lies at a higher level.  For example in the plan view of $ x_1 y_2$ the path of point $1$ lies below that of point $2$, giving views of $ x_1 y_2$ and $ y_2 x_1$ in figure \ref{pathproduct}.

\begin{figure}[ht]
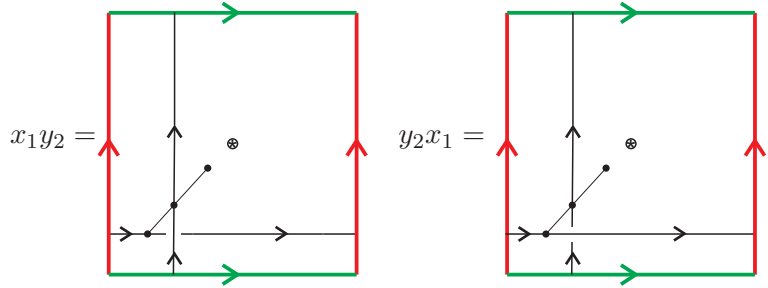

\bc $ x_1 y_2 =$\xioneetatwo $\quad  y_2 x_1 =$ \etatwoxione
\ec
\caption{Plan views of $ x_1 y_2$ and $ y_2 x_1$}\label{pathproduct}
\end{figure}

When two braids are composed there may be a path on the plan view that passes through a braid point at an intermediate stage. The plan can be altered to avoid such intermediate calls, by diverting the path slightly away from the braid point.
For example the braid $ x_1 y_1$ starts with a plan view in figure \ref{xy}. When the intermediate visit to braid point $1$ is diverted a plan view for $ x_1 y_1$ is shown in  figure \ref{smoothxy} along with   a view for $ y_1 x_1$.
\begin{figure}[ht]
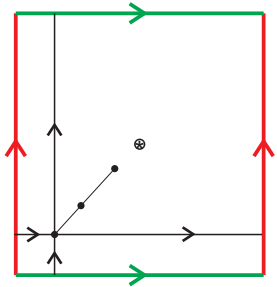

\bc \xioneetaone
\ec
\caption{Plan view of $ x_1 y_1$}\label{xy}
\end{figure}
\begin{figure}[ht]
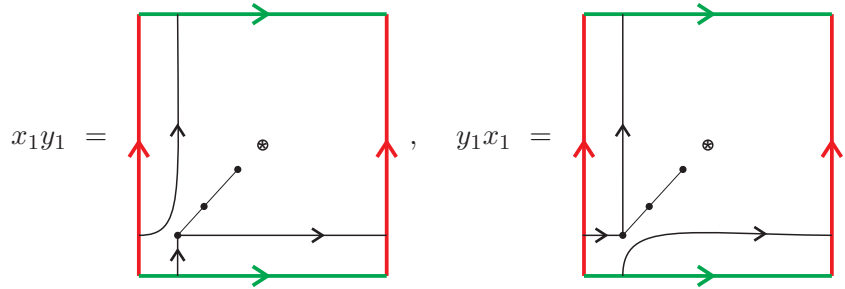

\bc $ x_1 y_1\  =\ $  \xioneetaonedivert\ , \quad $ y_1 x_1\  =\ $ \etaonexionedivert
\ec
\caption{Smoothed plan view of $ x_1 y_1$ and $ y_1 x_1$}\label{smoothxy}
\end{figure}

With further smoothing we get the plan view of the commutator $ x_1 y_1 x_1^{-1} y_1^{-1}$ as shown in figure \ref{xycommutator}. 
\begin{figure}[ht]
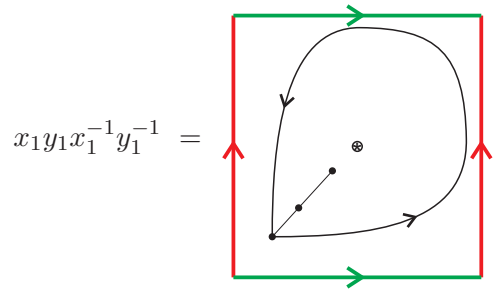

\bc $ x_1 y_1 x_1^{-1} y_1^{-1}\ =\ $ \xioneetaonecommutator 
\ec
\caption{Plan view of $ x_1 y_1 x_1^{-1} y_1^{-1}$}\label{xycommutator}
\end{figure}
From its elevation view in figure \ref{xycommutatorelev}
\begin{figure}[ht]
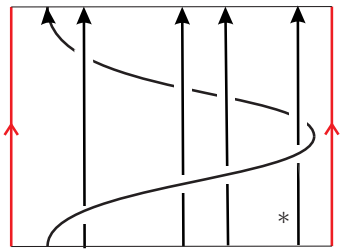

\bc
\labellist\small
\pinlabel{$*$} at 288 395
\endlabellist
\commutatoronebraid
\ec
\caption{Elevation of  $ x_1 y_1 x_1^{-1} y_1^{-1}$}\label{xycommutatorelev}
\end{figure}
we can write it as 
\[ x_1 y_1 x_1^{-1} y_1^{-1}=\sigma_1\sigma_2\cdots\sigma_{n-1}P\sigma_{n-1}\cdots\sigma_2\sigma_1. \] Here \[P=\Pbraid\] is the braid taking string $n$ once round the base string, with plan view \bc \Pplan\ec

This gives an expression 
\[P=\sigma_{n-1}^{-1}\cdots\sigma_1^{-1}x_1 y_1 x_1^{-1} y_1^{-1}\sigma_1^{-1}\cdots\sigma_{n-1}^{-1}.\]
as a braid in the punctured torus, in terms of the generators $x_1,y_1, \sigma_i$.

 As a further help in using the plan view for paths we can alter the view near the projection of one of the braid points, where a path starts out at the lowest level from the braid point and finishes at the highest level. Then another path crossing nearby (with either orientation) can be moved across the braid point as shown locally in figure \ref{braidpointmove}.
 \begin{figure}[ht]
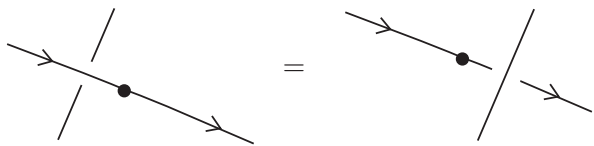

\bc \braidplanunder $\quad =\quad $ \braidplanover\ec
\caption{Moving an arc past a braidpoint}\label{braidpointmove}
\end{figure}

Apply this to the view of $ y_1 x_2$ by moving the path from braid point $1$ across braid point $2$. This gives \[
 y_1 x_2 =  \etaonexitwo = \xitwoalphatwo=  x_2\alpha_2\] where \[\alpha_2= \alphatwo =  \sigma_1^{2} y_1,\] and thus
\begin{eqnarray}
 x_2 y_1^{-1}&=& y_1^{-1} x_2\sigma_1^2.
\end{eqnarray}

 We can rewrite this equation in terms of the generators $x_1$ and $y_1$ as
\[\sigma_1^{-1}x_1\sigma_1^{-1}y_1^{-1}=y_1^{-1}\sigma_1^{-1}x_1\sigma_1\] and further
\[\sigma_1^{-2}x_1y_2^{-1}=y_2^{-1}x_1.\]

A similar argument, moving one path across braid points $2 \ldots n$,  shows that \[ y_1  x_2  x_3 \cdots  x_n = \etaonexitwon= \etaonexitwondivert= x_2  x_3 \cdots  x_n \alpha_n\] in the punctured braid group, where \[\alpha_n= \alphan= \beta_n  y_1,{\rm  with } \ \beta_n=\xioneetaonecommutatorbraid \] as in figure \ref{betan}, giving
\[ y_1 x_2 \cdots x_n= x_2\cdots x_n\beta_n y_1.\]

 \begin{figure}[ht]
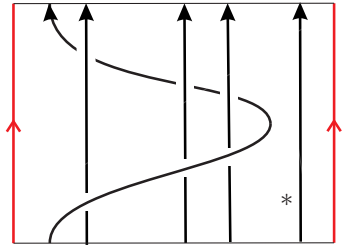

\bc
\labellist\small
\pinlabel{$*$} at 265 375
\endlabellist
\betanbraid
\ec
\caption{Side view of the braid $\beta_n=\sigma_1\sigma_2\cdots\sigma_{n-1}\sigma_{n-1}\cdots\sigma_2\sigma_1$}\label{betan}
\end{figure}

Bellingeri  \cite[theorem 1.1]{B04}  gives a presentation for the group of $n$-braids in the punctured torus  with generators \[\sigma_1,\cdots,\sigma_{n-1}, a,b,\] and relations
\begin{eqnarray}
\sigma_i\sigma_j&=&\sigma_j\sigma_i, |i-j|>1\\
\sigma_i\sigma_{i+1}\sigma_i&=&\sigma_{i+1}\sigma_i\sigma_{i+1}\\
\sigma_i a&=&a\sigma_i, i>1\\
\sigma_i b&=&b\sigma_i, i>1\\
a\sigma_1^{-1}a\sigma_1^{-1}&=&\sigma_1^{-1}a\sigma_1^{-1}a\\
b\sigma_1^{-1}b\sigma_1^{-1}&=&\sigma_1^{-1}b\sigma_1^{-1}b\\
b\sigma_1^{-1}a\sigma_1&=&\sigma_1^{-1}a\sigma_1^{-1}b
\end{eqnarray}

In our notation this corresponds to a presentation with generators $x_1, y_1,\sigma_i$ taking $a =y_1$ and $b=x_1^{-1}$ and $\sigma_i^{-1}$ in place of $\sigma_i$.

Bellingeri's relations involving $a$ and $b$ correspond to the equations 
\beqn
x_1 x_2&=& x_2 x_1\\
y_1 y_2 &=& y_2 y_1\\
x_2 y_1^{-1}&=& y_1^{-1} x_2\sigma_1^2
\eeqn
when written in terms of the generators $x_1, y_1, \sigma_1$.

\subsection{A presentation for the algebra $\bsk_n(T^2,*)$}

\begin{definition}\label{def:bskein}
The braid skein algebra $\bsk_n(T^2,*)$ is defined to be ${\bf Z}[s^{\pm1},c^{\pm1}]$-linear combinations of $n$-braids in the punctured torus, up to equivalence, subject to the local relations
\begin{equation}\label{eq:skeinrel}
\Xor -\Yor=(s-s^{-1})\Ior 
\end{equation}
and
\begin{equation}\label{eq:puncturecross}
\labellist\small
\pinlabel{$*$} at 75 105
\endlabellist\basecross \quad=\quad c^2\ 
\labellist\small
\pinlabel{$*$} at 80 100
\endlabellist\baseidentity
\end{equation}
between braids.
\end{definition}

By the term \emph{local relation} in this definition we mean that the braids in the relations only differ as shown inside a $3$-ball. We would like to find a ``small'' generating set for the ideal defined by these relations, which we do in the following three theorems. (To simplify exposition, Theorems \ref{thm:brskein} and \ref{thm:brpuncture} are proved in Subsection \ref{sec:isorels}.)

\begin{theorem}\label{thm:brskein} Suppose that $\alpha,\beta,\gamma$ are three $n$-braids in the punctured torus whose diagrams can be isotoped in $(T^2\oursetminus \{*\})\x I$, fixing the boundary, so that they differ only inside a ball as 
\[
\alpha=\Xor, \ 
\beta= \Yor, \ 
\gamma= \Ior.
\]


Then there exists a braid $L$  such that 
\[ L\beta=\sigma^{-2}L\alpha,\  L\gamma=\sigma^{-1}L\alpha.\]
\end{theorem}
\begin{theorem}\label{thm:brpuncture} Suppose that $\delta,\epsilon$ are two $n$-braids in the punctured torus whose diagrams can be isotoped in $(T^2 \oursetminus D^2)\x I$, fixing the boundary, so that they differ only in a ball as 
\[
\delta= \labellist\small
\pinlabel{$*$} at 75 105
\endlabellist\basecross \  ,\ 
\epsilon= \labellist\small
\pinlabel{$*$} at 80 100
\endlabellist\baseidentity \ .
\]

Then there exists a braid $L$ such that 
\[ 
L\epsilon=P^{-1}L\delta,
\]
where $P$ is the braid taking string $n$ once round the base string, shown here in plan and elevation. 
\bc $P$ \quad =\quad \Pplan \quad =\quad \labellist\small
\pinlabel{$*$} at 193 130
\endlabellist
\Pbraid \ec

\end{theorem}

\begin{theorem} The ideal generated by \eqref{eq:skeinrel} and \eqref{eq:puncturecross} is the same as the ideal defined by
\be
\sigma_1-\sigma_1^{-1}=(s-s^{-1})\label{eq:Hecke}
\ee
\be
P=c^2.\label{eq:basecircle}
\ee
\end{theorem}

\begin{proof}
Clearly the equations $\sigma_1-\sigma_1^{-1}=s-s^{-1}$ and $P=c^2$ are special cases of 
\eqref{eq:skeinrel} and \eqref{eq:puncturecross} .

Conversely given any braids $\alpha,\beta, \gamma$ whose diagrams differ in some ball as \[\alpha=\Xor,\beta=\Yor,\gamma=\Ior.\] 

By theorem \ref{thm:brskein} we can write \[L(\alpha-\beta) =(1-\sigma_1^{-2})L\alpha=(\sigma_1-\sigma_1^{-1})\sigma^{-1}L\alpha. \]

Then equation \eqref{eq:Hecke} shows that \[L(\alpha-\beta) =(s-s^{-1})\sigma_1^{-1}L\alpha=(s-s^{-1})L\gamma.\] Hence $\alpha-\beta=(s-s^{-1})\gamma$, and so  $\alpha, \beta$ and $\gamma$ satisfy equation \eqref{eq:skeinrel}.

To deduce equation \eqref{eq:puncturecross} for braids $\delta$ and $\epsilon$ as in theorem \ref{thm:brpuncture} write \[L\epsilon =P^{-1}L\delta\] and apply equation \eqref{eq:basecircle} to get \[L\epsilon =c^{-2}L\delta\] and hence
\[\delta=c^2\epsilon.\] 
\end{proof}

We can now adjoin these relations to Bellingeri's presentation for the braid group of the punctured torus to give a presentation of the algebra $\bsk_n(T^2,*)$.

\begin{theorem}\label{braidpresentation}
The algebra $\bsk_n(T^2,*)$ can be presented by the braids
\[\sigma_1,\cdots,\sigma_{n-1},x_1,y_1,\]
with relations
\begin{eqnarray}
\sigma_i\sigma_j&=&\sigma_j\sigma_i, |i-j|>1\label{eq:braidstart}\\
\sigma_i\sigma_{i+1}\sigma_i&=&\sigma_{i+1}\sigma_i\sigma_{i+1}\\
\sigma_i x_1&=&x_1\sigma_i, i>1\\
\sigma_i y_1&=&y_1\sigma_i, i>1\\
x_1\sigma_1^{-1}x_1\sigma_1^{-1}&=&\sigma_1^{-1}x_1\sigma_1^{-1}x_1\\
y_1\sigma_1y_1\sigma_1&=&\sigma_1y_1\sigma_1y_1\\
x_1^{-1}\sigma_1y_1\sigma_1^{-1}&=&\sigma_1y_1\sigma_1x_1^{-1}\label{eq:braidend}\\
(\sigma_1-s)(\sigma_1+s^{-1})&=&0 \label{eq:skeinone}\\
x_1y_1x_1^{-1}y_1^{-1}&=&c^2\sigma_1\sigma_2\cdots\sigma_{n-1}\sigma_{n-1}\cdots\sigma_2\sigma_1 \label{eq:commutator}
\end{eqnarray}
\end{theorem}

\begin{proof} 
In our notation Bellingeri's generators are $a=y_1, b=x_1^{-1}$, and our $\sigma_i$ is Bellingeri's $\sigma_i^{-1}$.

Relations \eqref{eq:braidstart} to \eqref{eq:braidend} then present the algebra of $n$-braids in the punctured torus, by \cite{B04}.
Relation \eqref{eq:skeinone} is equivalent to relation \eqref{eq:Hecke}.
Relation \eqref{eq:commutator} is equivalent to the relation \eqref{eq:basecircle}, $P=c^2$, since
\[ x_1 y_1 x_1^{-1} y_1^{-1}=\sigma_1\sigma_2\cdots\sigma_{n-1}P\sigma_{n-1}\cdots\sigma_2\sigma_1. \]
\end{proof}

\begin{remark}
As confirmation that our conventions are consistent with these relations note that with $x_2=\sigma_1^{-1}x_1\sigma_1^{-1}$ and $y_2=\sigma_1y_1\sigma_1$ the relations between the generators $x_1$ and $y_1$ become
$x_1x_2=x_2x_1, y_1y_2=y_2y_1$ and $y_2x_1^{-1}=x_1^{-1}y_2\sigma_1^{-2}$. These relations have already been demonstrated in our illustrations above.
\end{remark}



\begin{theorem}\label{thm:bskiso}
The skein algebra $\bsk_n(T^2,*)$ is isomorphic to the double affine Hecke algebra $\daha_n$.
\end{theorem}

\begin{proof}
We construct inverse homomorphisms between the two algebras.

\begin{itemize}
\item
 Define a homomorphism from $\bsk_n(T^2,*)$ to $\daha_n$ by
sending $x_1,y_1,\sigma_i$ to $X_1,Y_1, T_i^{-1}$ and $s^2,c^2$ to $t,q^{-1}$.

 To show that  this gives a homomorphism it is enough to check that the relations in the presentation of $\bsk_n(T^2,*)$ hold after the assignment of generators in $\daha_n$. 

The only relation for which this is not immediately clear is relation \eqref{eq:commutator} in  $\bsk_n(T^2,*)$. 
Relation \eqref{eq:commutator} can be written \[x_1y_1x_1^{-1}
=c^2 \beta_n y_1. \]
 We also know that 
\[y_1x_2\cdots x_n=x_2\cdots x_n \beta_n y_1.\] 
The relation can then be rewritten as 
\[c^{-2}x_1y_1x_1^{-1}=(x_2\cdots x_n )^{-1}y_1x_2\cdots x_n .\]

In our assignment to $\daha_n$ we can see that each $x_i$ is sent to $X_i$. It is then enough to check that  \[qX_1Y_1X_1^{-1}=(X_2\cdots X_n )^{-1}Y_1X_2\cdots X_n\] in $\daha_n$. This follows immediately from the last relation for $\daha_n$ and the fact that the elements $X_i$ all commute.  

\item
We can  define an inverse homomorphism  from $\daha_n$ to $\bsk_n(T^2,*)$ by sending $X_i, Y_i, T_i$ to $x_i,y_i,\sigma_i^{-1}$ and $t,q$ to $s^2,c^{-2}$. 
Our illustrations above confirm that the relations from $\daha_n$ hold in  $\bsk_n(T^2,*)$ after this assignment. 


\end{itemize}
\end{proof}

\subsection{Isotopies of relations}\label{sec:isorels}

  In this section we prove Theorems \ref{thm:brskein} and \ref{thm:brpuncture}. We have   reformulated the proofs so as to apply a theorem of Yi Ni, \cite[Theorem 1.1]{Ni11}, about Dehn surgery on a knot in a $3$-manifold $F\x I$ where $F$ is a surface with boundary.

   Dehn surgery on a curve $K$ in a $3$-manifold $M$ is the operation of removing a solid torus neighbourhood of $K$ and regluing a solid torus in its place, with a specification, in general by a rational number, of how the meridian of the solid torus is glued back.  When the surgery curve $K$ spans a disc in $M$ the manifold resulting from $\pm 1/n$ surgery on $M$ is homeomorphic to the original manifold $M$, and can be viewed as cutting open $M$ along the spanning disc and then regluing after making $n$ full twists on the disc.
   
Ni shows that when Dehn surgery on a curve $K$ in the product $M=F\x I$ results in a manifold $N$ which is homeomorphic  to $M$ in a controlled way then $K$ must lie in a very simple way in the product.  
  
   \begin{proof} [Proof of theorem \ref{thm:brskein}] We are given   diagrams for $\alpha, \beta$ and $\gamma$ which differ as shown inside a ball $D$.  We say that $\beta$ and $\gamma$ are given by \emph{switching} or \emph{smoothing} $ \alpha$ in $D$.
   
   We can specify the ball by including its oriented equator $K$  in the diagram of $\alpha$, to make a diagram $\alpha \cup K$ which includes $K$ as  a closed component. We call $\alpha \cup K$ a \emph{switching configuration} for $\alpha$ and $K$ the \emph{switching curve}, as shown in figure \ref{fig:switchingconfig}.  
  \begin{figure}[ht]
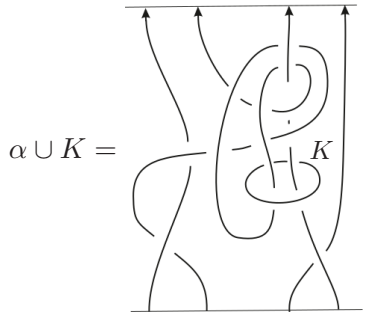

   \bc $\alpha\cup K =\labellist\small
\pinlabel{$K$} at 204 170
\endlabellist\switchconfig $ \ec
\caption{A switching configuration for a braid $\alpha$.} \label{fig:switchingconfig}
\end{figure}

   Using the switching configuration $\alpha\cup K$ has the advantage that the whole diagram can be manipulated by isotopy, while still allowing us to recover diagrams for $\beta$ and $\gamma$ as follows.
   The switching curve  $K$ spans an embedded $2$-disc which is crossed by $\alpha$ in just two points, in the same direction.  We can retrieve the ball $D$ where the switch and smooth takes place  as a neighbourhood of this disc, while the switched and smoothed diagrams  $\beta$ and $\gamma$ are given by replacing the two arcs crossing the neighbourhood  of the disc with  two other arcs as follows. 
   
   A switching curve $K$ around $\alpha$ is shown locally here, \bc $ \alpha\cup K \quad =\quad \labellist\small
\pinlabel{$K$} at 150 90
\endlabellist\switchcurve $\ . \ec  The switched  and the smoothed diagrams, in the neighbourhood of the disc spanning $K$, are then
   \bc $\beta\quad = \quad \switch, \qquad \gamma \quad = \quad \smooth $.\ec

  The effect of the switching operation which gives $\beta$ from $\alpha$ can then be described as performing $\pm 1$ Dehn surgery on the switching curve $K$, since it corresponds to making   a full twist on the strings crossing the disc which spans $K$. This observation means that a switching curve  is an important feature when considering the application of Ni's theorem.

  We can isotop a given switching configuration \[\alpha\cup K =\labellist\small
\pinlabel{$j$} at 220 00
\pinlabel{$i$} at 90 00
\pinlabel{$K$} at 204 170
\endlabellist\switchconfig \]\\[2mm] to push the switching curve to the bottom of the braid. To do this easily we first straighten out the braid strings to form the identity braid $\iota$,  by looking at 
  \[(\alpha\cup K)\alpha^{-1} = \iota \cup K'=\labellist\small
\pinlabel{$j$} at 237 0
\pinlabel{$i$} at 95 0
\pinlabel{$K'$} at 300 180
\endlabellist\identityconfig\ .\]  We can then redraw \[\alpha\cup K =\switchconfig\quad =\quad \labellist\small
\pinlabel{$j$} at 237 0
\pinlabel{$i$} at 95 0
\pinlabel{$\iota\cup K'$} at 290 125
\pinlabel{$\alpha$} at 275 300
\endlabellist\compositeconfig \qquad=(\iota\cup K')\alpha.\] \\[2mm] In this isotopic switching configuration the switching curve has been moved down to $K'$. After switching we  get \[\beta =\beta'\alpha\] where $\beta'$ is the result of switching the identity braid using the switching curve $K'$.  It follows that $\beta'=\beta\alpha^{-1}$ is then also a braid.

 We can now look at $K'$ in plan view as a knot diagram in $T^2-\{1,\ldots,n,*\}$.
     The diagram of $K'$  encircles points $i$ and $j$, where the original switching takes place between strings $i$ and $j$ of $\alpha$. 
    \begin{remark}
    We must have $i\ne j$ here, otherwise the smoothed curve $\gamma$ will have a closed component, and could not be itself a braid.
    \end{remark}
     The plan view of $K'$ will then    look something like \[\labellist\small
\pinlabel{$j$} at 237 100
\pinlabel{$i$} at 90 100
\pinlabel{$K'$} at 05 60
\endlabellist\switchplan\]
  in general.   The result of switching, $\beta'$, can be viewed in this projection  by taking the point $i$ through the  projected spanning disc and moving it out to the curve $K'$, then once round $K'$ and  back to the point $i$, while leaving the other strings alone, as indicated here,
  \[\labellist\small
  \pinlabel{$j$} at 237 100
\pinlabel{$i$} at 110 100
\endlabellist\switcheddiagram.\]   
   
  If there  are crossing points in the diagram of $K'$, as here,    the resulting  diagram of $\beta'$ seems unlikely in general to be the projection of a braid since the crossings for the projection of a braid will all appear as undercrossings at their first encounter round $K'$.
  Although it is possible that this could be realised after an isotopy we can  
   apply Ni's theorem \cite[Theorem 1.1] {Ni11} to prove the following result. 
  \begin{theorem} \label{th:switchcurve}
  If $\iota \cup K'$ is a switching configuration for the identity braid which gives rise to a braid after switching then  up to isotopy the diagram of $K'$ in $T^2-\{1,\ldots,n,*\}$ has no crossing points. 
  \end{theorem}
   \begin{proof}
  By removing a neighbourhood of each of the braid strings from $\iota$ we can regard $K'$ as a knot in the product manifold $F\x I$ with $F=T^2-(n+1)\  \mathrm{discs}$. We know that $\pm1$ Dehn surgery on $K'$ gives the exterior of the braid $\beta'$. Now any braid exterior is  homeomorphic to $F\x I$ by a level-preserving homeomorphism. Theorem 1.1 of \cite{Ni11} shows that $K'$ then has a diagram in $F$ with $0$ or $1$ crossing. The possibility of $1$ crossing can be excluded, since the spanning disc for $K'$ meets two strings of $\iota$ in the same sense. Consequently after isotopy the projection of $K'$ is  an embedded curve $C\subset F$ which bounds a $2$-disc in $F$ containing just two braid points.
  \end{proof}
  
  It is convenient to associate with any such  embedded curve $C\subset F$ a switching configuration $\iota\cup C^+$ for the identity braid, using $C^+=C\x\frac{1}{2}\subset F\x I$ as switch curve. For example the standard curve $C_{1\,2}$ around the braid points $1$ and $2$ gives the basic switching configuration
   \[\iota\cup C_{1\,2}^+ \quad = \quad \labellist\small
\pinlabel{$C_{1\,2}^+$} at 170 80
\pinlabel{$1$} at 30 50
\pinlabel{$2$} at 120 50
\endlabellist\standardswitchcurve  \]
  We know that switching with this particular configuration gives rise to  the braid $\sigma_1^{-2}$, and smoothing gives $\sigma_1^{-1}$. Similarly the standard embedded curve $C_{n\, *}$ around the braid point $n$ and the base point $*$ gives rise to the braid $P^{-1}$ after switching.
  
  By Theorem \ref{th:switchcurve} the projection $C$ of $K'$ is an embedded curve in $T^2$, which bounds a $2$-disc in $T^2$  containing two braid  points $i$ and $j$, for example
  \[C= \labellist\small
\pinlabel{$j$} at 190 100
\pinlabel{$i$} at 90 100
\endlabellist\embeddeddiagram \ ,\]  and hence we can write $\iota\cup K' =\iota\cup C^+$.
 
  We can  find an isotopy of $C\cup \{1,\ldots,n\}$ in $T^2-\{*\}$ which moves $C$ to the  standard curve $C_{1\,2}$ around points $1$ and $2$, and carries the set of braid points to itself. Such an isotopy can be constructed by first moving the points $i$ and $j$ close together within $C$ and then moving the whole disc to standard position,  and finally moving the remaining braid points.  The reverse of this isotopy  provides a braid $L$ with the property that \[(\iota\cup C^+_{1\,2})L =L(\iota\cup C^+)\] up to isotopy,  so in this sense $L$ intertwines the curve $C^+$ with the standard curve $C^+_{1\, 2}$.  
  As a result we have isotopic switching configurations for the braid $L\alpha$,
   \[  L(\alpha\cup K)=L(\iota\cup C^+)\alpha=(\iota\cup C^+_{1\,2})L\alpha\ , \] 
   in which the switching curve $K$ is moved down successively from $\alpha$ through $C$ to become the standard curve $C^+_{1\,2}$ at the bottom.
   
  By switching this configuration we have
  \[L\beta =\sigma_1^{-2}L\alpha \] 
  and by smoothing we get \[ L\gamma =\sigma_1^{-1}L\alpha ,\]
  as required for theorem \ref{thm:brskein}.

  \end{proof}
  
  \begin{proof} [Proof of theorem \ref{thm:brpuncture}] The braid $\epsilon$ arises from $\delta$ by a switch using a switching curve $K$ around the base string and string $i$.  As in the  proof of theorem \ref{thm:brskein} we can write $\delta\cup K = (\iota\cup K')\delta$ up to isotopy and apply  Theorem \ref{th:switchcurve} to show that $K'$ projects to an embedded curve $C$ in $T^2$  around the two points $i$ and $*$. This time find an isotopy of $T^2$ which carries $C$ to the standard curve $C_{n\, *}$ around the points $n$ and $*$, and use it provide a braid $L$ such that
  \[ L(\delta\cup K)=L(\iota\cup C^+)\delta =(\iota\cup C^+_{n\, *})L\delta \ .\]
  
 In this switching configuration for $L\delta$ use either $K$ or the isotopic curve $C_{n\,*}$ to switch. Switching $\iota\cup C^+_{n\, *}$ gives the braid $P^{-1}$ and switching $\delta\cup K$ gives $\epsilon$. Hence switching the whole configuration gives
  \[L\epsilon=P^{-1}L\delta \ ,\] as required for Theorem \ref{thm:brpuncture}. 
  
    \end{proof}

\section{The tangle skein algebra}
In this section we generalize the definition of the braid skein algebra using framed tangles, and we conjecture that this produces the same algebra as the braid skein algebra. The multiplication in the tangle skein algebra is the same as in the previous section, and comes from stacking in the $[0,1]$ direction. In the next section we use this conjecture to relate the classical skein algebra of closed links in the punctured torus to the elliptic Hall algebra.

In Homflypt skein theory we consider  oriented \emph{banded} curves in a $3$-manifold $M$, possibly with marked input and output points on its boundary.

Here are some such pieces 
\bc
\Idorband \qquad  \Xorband \qquad  \rcurlorband
\ec

We can think of these as made of flat tape rather than rope.
The only difference from rope is that the tapes can have extra twists in them such as
\bc\Idorleftband \qquad or \quad \Idorrightband\ec

 Twists may be dealt with by drawing little kinks in the diagram, replacing \bc\Idorrightband \qquad by \quad \rcurlorband\ec   and \bc\Idorleftband \qquad  by \quad \lcurlorband\ec

When there are boundary points the curves will include oriented arcs joining input to output points. In addition we can have some closed oriented curves.

The general Homflypt skein $\sk(M)$ is defined to be $\Z[s^{\pm1}, v^{\pm1}]$-linear combinations of banded links, up to isotopy, with the basic linear relations
\begin{align*}
	\Xorband\  -\ \Yorband \qquad &=\qquad{(s-s^{-1})}\quad\ \Iorband  \\
	\rcurlorband \qquad&=\qquad {v^{-1}}\quad \Idorband 
\end{align*}
between banded links whose diagrams differ only locally as shown.


Special cases of  interest to us are  where $M=F\x I$ for a surface $F$, with or without boundary. In such cases we write $\sk(F)$ for the skein $\sk(M)$, which has the structure of an algebra, with product induced by stacking curves in the direction of the interval $I$.

In \cite{MS17} we have looked at the case where $F=T^2$, and given a presentation for $\sk(T^2)$. 

The case $\CC=\sk(A)$, where $A$ is the annulus, is a commutative algebra. It has been widely studied,  originally by Turaev \cite{Turaev}, and subsequently by Morton and others.

In our present work we will incorporate the skein of  the torus with one hole, $\sk(T^2 \oursetminus  D^2)$, including elements which map to the generators of $\sk(T^2)$ under the homomorphism induced by the inclusion $T^2\oursetminus D^2 \to T^2$.

Again in the case $M=F\x I$ we will consider the case where we fix $n$ input points in $F\x \{0\}$, and take the corresponding $n$ output points in $F\x \{1\}$. Stacking in the $I$ direction will give  this skein the structure of an algebra over $\Z[s^{\pm1}, v^{\pm1}]$ which we denote by $\sk_n(F)$.

The simplest case of this, when $F=D^2$, gives the algebra $\sk_n(D^2)$. This algebra is a version of the Hecke algebra $H_n(z)$ of type $A$, based on the quadratic relation $\sigma_i^2=z\sigma_i +1$, where $z=s-s^{-1}$.

In anticipation of the next section we are led to consider the skein $\sk_n(T^2 \oursetminus \{*\})$ of the  punctured torus. In order to incorporate our algebra $\bsk_n(T^2,*)$ into this framework we will adjoin the relation
\bc
\labellist\small
\pinlabel{$*$} at 146 740
\endlabellist\Xor$\quad=\quad c^2\ $
\labellist\small
\pinlabel{$*$} at 146 740
\endlabellist\Yor
\ec
to allow a string to cross through the fixed string $\{*\}\x I$ in $T^2\x I$ which defines the puncture.
With this extra relation in place we  use the notation $\sk_n(T^2,*)$ for the resulting algebra over $\Z[s^{\pm1}, v^{\pm1}, c^{\pm 1}]$.

\begin{theorem}\label{thm:dahatoskein} There is an algebra homomorphism 
	\[F_n: \bsk_n(T^2,*)\cong \ddot H_n\to \sk_n(T^2,*).\]
\end{theorem}
\begin{proof} The homomorphism $F_n$ is defined on a braid by making a consistent choice of framing for it.
Braids in $T^2$ can be framed by fixing a direction in $T^2$, say  the $(1,0)$ direction, and taking the band on each braid string in this direction, which is  transverse to the string at all points. This appears to give the framing used in \cite{BWPV14}.  Under any braid isotopy the bands will be preserved, and the relations between braids will satisfy the skein relations between banded curves.
\end{proof}

In this section and the next we consider the algebra $\sk_n(T^2,*)$, which incorporates general framed tangles along with closed curves, with the elliptic Hall algebra and its relation to the algebras $\ddot H_n$ in mind. One outcome of this is the following result, which will be established in the next section.



	\begin{theorem} The homomorphism \[F_n: \bsk_n(T^2,*)\to \sk_n(T^2,*)\]
	is surjective. \label{surjection}
	\end{theorem}
	\begin{proof}
	
	We must look at elements  of the skein $\sk_n(T^2,*)$ and reduce them to combinations of braids. In particular we have to deal with closed curves as well as braids. 
	
		Choose a disc $D^2$ in $T^2$ which contains the $n$ braid points and the puncture $*$. A suitable choice is a neighbourhood of the line through the braid points $1,\cdots,n$ and the base point $*$, shown in figure \ref{disc}.
		The inclusion of $T^2-D^2$ in $T^2$, combined with the identity braid on the $n$ strings, induces an algebra homomorphism
	\[\varphi_n: \sk(T^2-D^2) \to \sk_n(T^2,*).\]

\begin{figure}[ht]
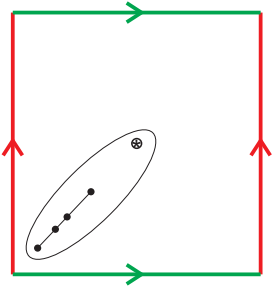

\bc
\Torusbasedisc\ec
\caption{A choice for the disc $D^2$ in $T^2$} \label{disc}
\end{figure}

	The  closed curves we   particularly want to use can be described in the skein $\sk(T^2-D^2)$.  This skein is an algebra with a homomorphism \[\sk(T^2-D^2)\to \sk(T^2)\] induced by filling in the disc. A presentation of the algebra $\sk(T^2)$ is given in \cite{MS17}, with generators $W_\xx, \xx\in \Z^2$, and relations \[[W_\xx,W_\yy]=(s^m-s^{-m})W_{\xx+\yy}\] where $m=\det(\xx \yy)$.
	
	When $\zz\in\Z^2$ is primitive there is an embedded curve $W_\zz\subset T^2$ which is taken to represent the element $W_\zz\in\sk(T^2)$. The same curve will give a well-defined element of $\sk(T^2-D^2)$, which we will also write as $W_\zz$. 
	Write $A_\zz\subset T^2-D^2$ for the annulus with $W_\zz$ as its core. Further elements $W_\xx$, where $\xx=m\zz$, are defined in section \ref{sec:closed} as in \cite{MS17} by suitably chosen combinations of curves in $A_\zz$, and determine the generators of $\sk(T^2)$ above.
	
	The elements $W_\xx$ do not generate the whole of the algebra $\sk(T^2-D^2)$, nor do all the commutation relations from \cite{MS17} hold in $\sk(T^2-D^2)$. The main problem is that the disc $D^2$ gets in the way of isotopies that can be made in $T^2$.
	
	\begin{remark}\label{rmk:ptorusrel}
	It is worth noting however that if $\yy$ and $\zz$ are	primitive, and the curves $W_\yy$ and $W_\zz$ intersect in just one point, then the commutation relation 
	\[ [W_\xx,W_\yy]=(s^m-s^{-m})W_{\xx+\yy}\] holds for $\xx=m\zz$. This is because the argument from \cite{MS17} only involves curves in the union of the annuli $A_\yy$ and $A_\zz$. 
	Hence in particular we have\footnote{The sign difference between the right hand side of \eqref{eq:demoskein}  and the corresponding relation in \cite{MS17} comes from the fact that our convention for the product in the skein algebra here (left element goes below the right element) is the opposite from \cite{MS17}.} 
	\begin{equation}\label{eq:demoskein}
	[W_{(m,0)}, W_{(0,1)}]=-(s^m-s^{-m})W_{(m,1)}
	\end{equation}
	in the image of $\daha_n$ in $\sk_n(T^2,*)$. Note that under the map\footnote{The existence of this map depends on the assumption that Conjecture \ref{conj:mainiso} is true.} $\sk^+(T^2-D^2) \to \E_{q,t}$ described in Theorem \ref{thm:psktoeha}, we have the assignments 
	\[W_{m,0} \mapsto \{m\}_s u_{m,0},\quad \quad W_{0,1} \mapsto \{1\}_s u_{0,1},\quad \quad W_{m,1} \mapsto \{1\}_s u_{m,1}
	\]
	where we have used the notation $\{d\}_s := s^d-s^{-d}$. Expanding the relation \eqref{eq:hallrel} in the case $\yy=(m,0)$ and $\xx = (0,1)$, we get
	\begin{equation}\label{eq:demoeha}
	[u_{m,0},u_{0,1}] = -u_{m,1}
	\end{equation}
	Thus, we see that that relation \eqref{eq:demoskein} in $\sk^+(T^2-D^2)$ gets mapped to the relation \eqref{eq:demoeha} in $\E_{q,t}$ under the algebra map in Theorem \ref{thm:psktoeha} (whose existence depends on Conjecture \ref{conj:mainiso}). 
	\end{remark}


	To prove theorem \ref{surjection} we first show that for each $n$ the image of $\sk(T^2-D^2)$ is represented by braids.  In other words we establish
	\begin{lemma} \[\varphi_n(\sk(T^2-D^2))\subset F_n(\bsk_n(T^2,*)).\] \label{braidrep}
	\end{lemma}
	
	This depends on two results.
	\begin{itemize}
	\item We can write any curve $C\subset T^2-D^2$ as a polynomial in the elements $\{W_\xx\}$ plus a linear combination of braids.
	\item Each $W_\xx$ is a linear combination of braids.
	\end{itemize}
	
	The proof of theorem \ref{surjection} is completed by showing that a general tangle can be written as a product of braids and closed curves which avoid $D^2$.

	\begin{lemma} The algebra $\sk(T^2-D^2)$ is generated by totally ascending single curves $C\subset T^2-D^2$. 
	\label{ascending}
	\end{lemma}
	\begin{proof} This is a standard skein theory exercise. We proceed by induction on the number of crossings in a diagram in $T^2-D^2$. Order the components of the diagram and choose a starting point on each component. A totally ascending diagram is one in which each crossing appears first as an undercrossing, when working along the component from the chosen starting point. Go through the components in turn, switching crossings, and using the skein relation, to end up with a diagram $L$ in which the components are totally ascending, along with a linear combination of diagrams with fewer crossings.
	
	The components of $L$ are then stacked one above the other, and so represent the product of single totally ascending curves.
	\end{proof}
	
	The ability to alter the starting point of a totally ascending diagram of a curve, using induction on the number of crossings in the diagram, will prove useful in the  arguments which follow.   
	
	\begin{lemma} Suppose that $C$ and $D$ are two diagrams differing only in the signs of their crossings, then $D=C+(s-s^{-1})\sum \pm D_\alpha$ where each $D_\alpha$ is a $2$-component diagram with fewer crossings than $C$. \end{lemma}
	\begin{proof} This follows immediately, using the skein relation at each crossing of $C$ to be switched.
	\end{proof}
	
	\begin{lemma} Let $C$ be a curve in $T^2-D^2$ which is totally ascending from a starting point immediately beside the boundary of $D^2$. The curve $C'$ given by diverting $C$ around the other side of $D^2$, which is also totally ascending and is isotopic to $C$ in $T^2$ but not in $T^2-D^2$, satisfies the relation
	\[\varphi_n(q^{\pm1}C'-C)=\pm(s-s^{-1})\sum_{i=1}^n \beta_i(C)\]
	where $\beta_i(C)$ are braids. \label{diversion}
	\end{lemma}
	\begin{proof}
	Divert the curve $C$ successively past the braid points to give curves $C=C_0$, $C_i$, $C_n$, and then $C'$, with $C_i$ crossing $D^2$ between points $i$ and $i+1$, interpreting $*$ as point $n+1$.
	The skein relation gives $C_i=C_{i-1}\pm (s-s^{-1})\beta_i(C)$ where $\beta_i(C)$ is a braid in which only string $i$ moves, because the curve $C_{i-1}$ is totally ascending at  the point of crossing string $i$.
	The result follows, given that $C'=q^{\pm 1}C_n$.
	\end{proof}
	\begin{remark} The sign depends on the direction of the curves $C_i$ across the disc. It will be $+$ if $C$ crosses the line of braid points in the $x$ direction. The curve $C$ represents an element  $w_C(x,y)\in \pi_1(T^2-D^2)$  based at the starting point. This fundamental group is the free group on $2$ generators, and the braid $\beta_1(C)$ is then $\beta_1(C)=w(x_1,y_1)$ while the successive braids $\beta_i(C)$ satisfy
	\[\beta_i(C)=\sigma_{i-1}\cdots \sigma_1\beta_1(C)\sigma_1\cdots\sigma_{i-1}.\]
	\end{remark}
	

%
We are now in a position to prove that a curve $C\subset T^2-D^2$ can be written in $\sk_n(T^2,*)$ as a polynomial in $\{W_\xx\}$ modulo $\im(F_n)$. Work by induction on the number of crossings in the diagram of $C$. We can assume by lemma \ref{ascending} that $C$ is totally ascending.

Write $\cc\in \Z^2$ for the  homology class of $C$ in $H_1(T^2)$, and write $\cc=m\zz$ with $\zz$ primitive. If $\cc=(0,0)$ choose any primitive as $\zz$. 

Fix a simple closed curve $Z$ in $T^2$ in the direction of $\zz$. Now perform a sequence of moves   at the expense of braids and other elements $W_\xx$, at each stage replacing $C$ by $C'$ homologous to $C$. Firstly move strands of $C$ across $D^2$ until there are no strands of $C$ lying between $Z$ and $D^2$ to one side of $Z$. At each move we  need to change the starting point to lie on the strand  to be moved across $D^2$, using lemma \ref{diversion}.

Now suppose that $Z$ is crossed  $\l$ times by $C$ in the direction away from $D^2$. If $\l=0$ then $C$ lies entirely in the annulus $A_\zz$ on the side away from $D^2$. It then represents an element in the skein of this annulus, $\sk(A_\zz)$,  which can be written as a polynomial in the commuting elements $\{W_{k\zz}\}$.  

Otherwise  $Z$ is also crossed  $\l$ times by $C$ in the opposite sense, because $C$ is homologous to $m\zz$. Choose a crossing in the direction away from $D^2$ where the next crossing is towards $D^2$, and take this as the starting point for $C$, again using the induction on the number of crossings. The arc of $C$ between the starting crossing and the next one   can now be isotoped across $Z$, without crossing $D^2$,  to reduce $\l$.

%


  In equation \eqref{eq:Wasbraid}, it is shown that $W_{(m,0)}$ can be written as a linear combination of braids, and combining this with the $SL_2(\Z)$ action shows that any curve in $T^2-D^2$ can be represented in $\sk_n(T^2,*)$ by a linear combination of braids. 
 This completes the proof of lemma \ref{braidrep}.


The last step in the proof of theorem \ref{surjection} is to deal with a general framed $n$-tangle in $(T^2,*)$. This will consist of $n$ arcs along with some closed curves. We use the plan view, modified slightly to separate the top and bottom points, and work by induction on the number of crossing points in a diagram to show that it represents a polynomial in braids and closed curves avoiding $D^2$. 

Make the tangle diagram totally ascending, choosing starting points first at the bottom point of each arc and then on the closed curves. This can be done by the induction. The resulting diagram represents a product of a braid with   closed curves, since the arcs are each totally ascending, and the crossings with closed curves all lie above them.
It remains to show that a single totally ascending curve $C$ in the torus which avoids the $n+1$ points $\{1,\cdots,n,*\}$ can be altered at the expense of braids to avoid the line through the $n+1$ points which determines $D^2$. 

Suppose that $C$ crosses the connecting line immediately to the right of point $i$. Make $C$ totally ascending at this crossing position, and then move $C$ across $i$ to lie immediately to its left, at the expense of a braid $\beta_i(C)$ as in the proof of lemma \ref{diversion}.  Continue moving intersections to the left, and eventually past the point $1$ at the end of the line, to finish by avoiding the connecting line altogether, and hence lying in $T^2-D^2$.  Lemma \ref{braidrep} then completes the proof of theorem \ref{surjection}
\end{proof}

\begin{remark} It is not clear whether the homomorphism $F_n$ is injective. 
 There can be the question of possible further relations between elements in the image of $\bsk_n(T^2,*)\cong \ddot H_n$ coming from the additional closed curves that can be used in $\sk_n(T^2,*)$.
\end{remark}

Despite the previous remark, we conjecture (see Conjecture \ref{conj:mainiso} in the introduction) that the algebra map $F_n:\ddot H_n\to \sk_n(T^2,*)$ in Theorem \ref{thm:dahatoskein} is an isomorphism.

%
%

\section{Relations with the elliptic Hall algebra}\label{fullskein}

In \cite{SV11} the authors relate the double affine Hecke algebras $\ddot{H}_n$ to the elliptic Hall algebra.
As part of their construction they make use of the sums of powers \[
\sum_i X_i^l, \sum_i Y_i^l \in \ddot{H}_n 
\] 
which have a  very useful skein theoretic description, and which led us to try including closed curves in our skein $\bsk_n(T^2,*)$. We will show that the images of these elements in $\sk_n(T^2,*)$ agree with the images of certain natural elements in $\sk(T^2\oursetminus D^2)$. In Theorem  \ref{thm:psktoeha}, we combine this with results of Schiffmann and Vasserot to show that Conjecture \ref{conj:mainiso} implies a weakened version of Conjecture \ref{conj:punctoeha}.

\subsection{Certain closed curves}\label{sec:closed}
For the moment consider the Homflypt skein $\sk_n(A)$ where $A$ is an annulus, using oriented diagrams in the thickened annulus $A\times I$ with $n$ output points on the top $A\times\{1\}$, and $n$ matching input points on $A\times\{0\}$.  We also allow closed components in the diagrams.

When restricted to braid diagrams    the skein $\bsk_n(A)$ is used by Graham and Lehrer as a model for the affine Hecke algebra $\dot {H}_n$, where composition is again induced by composition of braids.

Write $Z_i$  and $\overline Z_i$ for the elements represented in $\sk_n(A)$ by the diagrams shown here. Take the framing of the closed component as given by the plane of the diagram. 
\bc $Z_i\quad = \quad $\labellist\small
\pinlabel {$i$} at 210 335
\endlabellist \xiibraidblue\ , $\quad \overline Z_i\quad = \quad $\labellist\small
\pinlabel {$i$} at 210 335
\endlabellist \etaibraidblue\\[4mm]
\ec

It is readily established that 
\beqn \backid \quad -\quad  \frontid &=&(s-s^{-1}) \sum Z_i\\
&=&(s-s^{-1}) \sum \overline Z_i
\eeqn
 
 There is also a well-established element $P_m$ for each $m$ in the skein $\sk(A)=\CC$ of   the annulus with no boundary points  which satisfies the relation
 
 \beqn \labellist\small
\pinlabel {$P_m$} at 150 415
\endlabellist\backidblue \quad - \quad \labellist\small
\pinlabel {$P_m$} at 150 415
\endlabellist\frontidblue &=&  (s^m-s^{-m})\sum Z_i^m\\
&=&(s^m-s^{-m}) \sum \overline Z_i^m .
\eeqn
 A detailed account of $P_m$ can be found, for example, in \cite{Mor02b}.

When we embed $A$ into $T^2$ around the $(1,0)$ curve,  matching the braid points suitably, the induced homomorphism from $\sk_n(A)$ to $\sk_n(T^2,*)$ gives the equation
 \beqn(s^m-s^{-m})\sum_i x_i^m&=& \labellist\small
\pinlabel {$P_m$} at 160 353
\endlabellist\uxfront \quad-\quad\labellist\small
\pinlabel {$P_m$} at 160 387
\endlabellist\uxback \\
 &=&(1-c^{2m}) \ \labellist\small
\pinlabel {$P_m$} at 160 353
\endlabellist\uxfront
\eeqn  
 
 Similarly, taking $A$ around  the $(0,1)$ curve on $T^2$ we get
 \beqn(s^m-s^{-m})\sum_i y_i^m&=& \labellist\small
\pinlabel {$P_m$} at 135 405
\endlabellist\uyfront \quad-\quad\labellist\small
\pinlabel {$P_m$} at 105 405
\endlabellist\uyback \\
 &=&(c^{-2m}-1) \ \labellist\small
\pinlabel {$P_m$} at 135 405
\endlabellist\uyback
\eeqn 

In  $\sk_n(A)$, taking account of the crossing signs, we also have

 \beqn \labellist\small
\pinlabel {$P_m$} at 130 415
\endlabellist\backidleftblue \quad - \quad \labellist\small
\pinlabel {$P_m$} at 130 415
\endlabellist\frontidleftblue &=&  -(s^m-s^{-m})\sum Z_i^{-m}\\
&=&-(s^m-s^{-m})\sum \overline Z_i^{-m}.\eeqn

Placing $A$ along the $(1,0)$ curve then gives 
\beqn -(s^m-s^{-m})\sum x_i^{-m}&=& \labellist\small
\pinlabel {$P_m$} at 160 353
\endlabellist\uxfrontleft \quad-\quad\labellist\small
\pinlabel {$P_m$} at 160 387
\endlabellist\uxbackleft \\
 &=&(1-c^{-2m}) \ \labellist\small
\pinlabel {$P_m$} at 160 353
\endlabellist\uxfrontleft\ ,\\
\eeqn  

while placing $A$ along the $(0,1)$ curve gives
 \beqn-(s^m-s^{-m})\sum_i y_i^{-m}&=& \labellist\small
\pinlabel {$P_m$} at 135 405
\endlabellist\uyfrontdown \quad-\quad\labellist\small
\pinlabel {$P_m$} at 105 405
\endlabellist\uybackdown \\
 &=&(c^{2m}-1) \ \labellist\small
\pinlabel {$P_m$} at 105 405
\endlabellist\uybackdown\\
\eeqn

In \cite{SV11} there is a description of the elliptic Hall algebra which involves generators $u_\xx$ for every non-zero $\xx\in \Z^2$. These elements satisfy  certain commutation relations, and the comparison with the algebras $\ddot{H}_n$ requires the prescription of an image for each $u_\xx$, and a check on their commutation properties.

We can give a version of this comparison by using the skein $\sk_n(T^2,*)$, and the skein $\sk(T^2 -D^2)$.
Fix a disc $D^2$ in $T^2$ which includes the braid points and the base point. A suitable choice for our purposes is a neighbourhood of the  diagonal in the square. 
In the previous section we introduced the homomorphism
\begin{equation}\label{eq:fillinpunc}
\varphi_n: \sk(T^2-D^2)\to \sk_n(T^2,*)
\end{equation}
defined by taking the banded curves in $T^2 - D^2$ along with the identity $n$-braid in $\sk_n(T^2,*)$, consisting of $n$ vertical strings in $D^2\x I$ and the base string.

Now any oriented embedded curve in $T^2 - D^2$ is determined up to isotopy  by a primitive element $\zz\in \Z^2$, representing the homology class of the curve. This curve, framed by its neighbourhood in $T^2$ defines an element $W_\zz\in \sk(T^2-D^2)$.
For any other non-zero $\xx\in \Z^2$ write $\xx=m\zz$ with $m>0$ and $\zz$ primitive, and define $W_\xx$ to be $W_\zz$ with the closed curve decorated by the element $P_m$. 

We will write $W_\xx$ also for its image in the skein $\sk_n(T^2,*)$.
We then have plan views of $W_{(\pm m,0)}$ and $W_{(0,\pm m)}$ as
\bc $W_{(m,0)}\ =\ $ \labellist\small
\pinlabel {$P_m$} at 160 353
\endlabellist\uxfront
\ ,\quad $W_{(-m,0)}\ =\ $ \labellist\small
\pinlabel {$P_m$} at 160 353
\endlabellist\uxfrontleft
\ ,  
\ec
\bc $W_{(0,m)}\ =\ $ \labellist\small
\pinlabel {$P_m$} at 105 405
\endlabellist\uyback\ , \quad $W_{(0,-m)}\ =\ $ \labellist\small
\pinlabel {$P_m$} at 105 405
\endlabellist\uybackdown\ ,
\ec

Our equations above show that 
\begin{align}  
(1-c^{2m})W_{(m,0)}&=(s^m-s^{-m})\sum x_i^m,\label{eq:Wasbraid}\\
(c^{-2m}-1)W_{(-m,0)}&=(s^m-s^{-m})\sum x_i^{-m}\notag\\
(c^{-2m} - 1)W_{(0,m)}&=(s^m-s^{-m})\sum y_i^m\notag \\
(1-c^{2m})W_{(0,-m)}&=(s^m-s^{-m})\sum y_i^{-m}.\notag
\end{align}

\subsection{Comparison with the algebraic approach} 

For non-zero $\xx\in \Z^2$ Schiffman and Vasserot in \cite{SV11} define elements $Q_\xx$ in the spherical algebra $S\daha_n$, where $S\daha_n$ is defined as $e_n\daha_n e_n$, with $e_n\in H_n$ being the symmetrizer. They use the elements $Q_\xx$  in setting up their comparisons with the elliptic Hall algebra.  

Using the identification of $\bsk_n(T^2,*)$ with $\daha_n$, where $q=c^{-2}, s^2=t$, we show now that our elements $W_\xx$ are closely related to $Q_\xx\in S\daha_n$, when mapped into the full skein algebra $\sk_n(T^2,*)$.


Before doing this we note the construction of the symmetrizer $e_n\in H_n\subset \daha_n$ in the braid skein setting, as used by Aiston and Morton in \cite{AM98}.

We use the model of the Hecke algebra $H_n$ described in   \cite{MT90}, and further in \cite{AM98}. The symmetrizer is given there as a multiple of the quasi-idempotent 
 $a_n=\sum s^{l(\pi )}\omega _{\pi}$, where $\omega _{\pi}$ is the positive permutation braid associated to the permutation $\pi$  with length $l(\pi )$ in the symmetric group. The symmetrizer is then $e_n=\frac{1}{\alpha_n}a_n$  where $\alpha_n$ is given by the equation $a_na_n=\alpha _na_n$ \cite{{Luk05}, {AM98}}.  Using the  quasi-idempotent $b_n=\sum  (-s)^{-l(\pi )}\omega _{\pi}$ in a similar way gives the antisymmetrizer. 
 
 
 We  prefer to avoid the notation $S$ for the symmetrizer, because of  conflict with the notation for the symmetric group.  In \cite{SV11} the element $a_n$ is denoted by $\tilde S$, and the symmetrizer by $S$.  

\begin{theorem} \label{PWcomparison} For   $\xx\in\Z^2$ we have the following equality in $\sk_n(T^2,*)$:
\[
(q^m-1)e_nW_\xx e_n=(s^m-s^{-m})Q_\xx,
\] 
where $\xx=m\yy$ with $\yy$ primitive and $m>0$.
\end{theorem}

\begin{proof} We start from the definition in \cite{SV11} which sets $Q_{(0,m)}:=e_n\sum Y_i^m e_n$ for $m>0$. 

Our third equation above  proves the theorem for $\xx=(0,m)$, since \[(q^m-1)e_n W_{(0,m)} e_n= (s^m-s^{-m}) e_n\sum Y_i^me_n = (s^m-s^{-m}) Q_{(0,m)}.\]

When $\xx= ({\pm m,0}),  ({0, -m})$  the values of $Q_\xx$ are shown in  \cite[Eq. 2.16-2.18]{SV11} to be
\beqn    Q_{(-m,0)}&=& e_n\sum X_i^{-m}e_n\\
Q_{(0,-m)}&=& q^me_n\sum Y_i^{-m}e_n\\
Q_{(m,0)}&=& q^me_n \sum X_i^me_n.
\eeqn

The theorem follows immediately from \eqref{eq:Wasbraid} in these cases too, since \[(q^m-1)W_{(-m,0)}=(s^m-s^{-m})\sum x_i^{-m},\] giving the case $\xx=(-m,0)$, while \[(q^m-1)W_{(m,0)}=(s^m-s^{-m})q^m\sum x_i^{m}\] and \[(q^m-1)W_{(0,-m)}=(s^m-s^{-m})q^m\sum y_i^{-m},\] giving the other two cases.

We use automorphisms of $\daha_n$, and their counterpart in the skein models $\bsk_n(T^2,*)$ and $\sk_n(T^2,*)$ to establish the proof for general $\xx$.

Firstly, in our skein model, a right-hand Dehn twist about the (unoriented) $(1,0)$ curve in $T^2 - D$ induces an automorphism $\tau_1$ of $\sk(T^2-D^2)$, which carries $W_{\xx}$ to $W_\yy$ with 
\[
\yy=
\begin{pmatrix} 1&1\\ 0&1 \end{pmatrix}\xx.
\]  

A left-hand Dehn twist about the unoriented $(0,1)$ curve in $T^2 - D$ induces an automorphism $\tau_2$ of $\sk(T^2-D^2)$, which carries $W_{\xx}$ to $W_\yy$ with \[\yy=\begin{pmatrix} 1&0\cr 1&1 \end{pmatrix}\xx.\]  

These two automorphisms generate all homeomorphisms of $T^2$ which fix $D$, up to isotopy fixing $\partial D$. This group of automorphisms is isomorphic to the braid group $B_3$ with $\tau_1$ and $\tau_2^{-1}$ playing the roles of the usual Artin generators $\sigma_1, \sigma_2$. The kernel of the map to $SL(2,\Z)$ is infinite cyclic, generated by $(\tau_1 \tau_2^{-1}\tau_1)^4$, which is the right-hand Dehn twist about $\partial D$.

For any $\xx$ with $d(\xx)=m>0$ we can find an automorphism $\gamma$ so that $\xx=\gamma((0,m))$.

Now the effect of $\tau_1$ on the generators $\sigma_i, x_i, y_i$ of $\bsk_n(T^2,*)$ is
\beqn \tau_1(\sigma_i)&=& \sigma_i\\
\tau_1(x_i)&=&x_i\\
\tau_1(y_i)&=& \eta_i
\eeqn where 
\[ \eta_i=y_i x_i \delta_i\] and \[\delta_i= \sigma_{i-1}\ldots \sigma_1 \sigma_1\ldots \sigma_{i-1}.\]

The effect of $\tau_2$ is
\beqn \tau_2(\sigma_i)&=& \sigma_i\\
\tau_2(x_i)&=&\xi_i\\
\tau_2(y_i)&=& y_i
\eeqn where 
\[ \xi_i=x_i y_i \delta_i^{-1}.
\]

The automorphisms $\rho_1$ and $\rho_2$ used in \cite{SV11} agree with $\tau_2$ and $\tau_1$, given the correspondence of $x_i, y_i, \sigma_i$ with $X_i, Y_i, T_i^{-1}$ respectively.

Since $Q_\xx$ is given from $Q_{(0,m)}$ by applying a suitable product of $\rho_1$ and $\rho_2$, the same automorphism will carry $W_{(0,m)}$ to $W_\xx$ and the theorem will follow.
\end{proof}

\subsection{Without the symmetrizer}

Theorem \ref{PWcomparison}, which refers to elements of $\sk_n(T^2,*)$, suggests that $Q_\xx$ could be defined 
 unambiguously from an element $\tilde Q_\xx$ in   $\daha_n\cong \bsk(T^2,*)$ before passing to $S\daha_n$.  The kernel of the map from $B_3$ to $SL(2,\Z)$ is generated by $(\tau_1 \tau_2^{-1}\tau_1)^4$. In the skein model this is a Dehn twist about the boundary of the disc $D$, and so in this model we expect the following theorem, which we can prove algebraically.
  
\begin{proposition}
For any  $Z\in \daha_n\cong \bsk_n(T^2,*)$ we have \[(\tau_1 \tau_2^{-1}\tau_1)^4 Z=\Delta^{-2} Z \Delta^2,\]
where $\Delta^2$ is the full twist braid in the finite Hecke algebra $H_n$.  
\end{proposition}
\begin{proof} It is enough to prove this when $Z=x_1$ and $Z=y_1$, since these elements, along with $\sigma_i$, generate $\bsk_n(T^2,*)$. In the case $Z=\sigma_i$ we have $\tau_1(\sigma_i)=\tau_2(\sigma_i)=\sigma_i$, while the full twist $\Delta^2$ commutes with each $\sigma_i$.

We also know that 
\beqn
 \tau_1(x_1)&=&x_1\\
\tau_1(y_1)&=& y_1 x_1\\
\tau_2^{-1}(x_1)&=&x_1 y_1^{-1}\\
\tau_2^{-1}(y_1)&=&y_1
\eeqn

Writing $\tau_1 \tau_2^{-1}\tau_1=\theta$ we get
\[ \theta(x_1)=y_1^{-1}, \theta (y_1)=y_1 x_1 y_1^{-1}\] so
\[ \theta^2(x_1)=(\theta (y_1))^{-1}=y_1 x_1^{-1} y_1^{-1}=(y_1x_1)x_1^{-1}(y_1x_1)^{-1}\]
\[
\theta^2(y_1)=\theta (y_1)\theta (x_1)(\theta (y_1))^{-1}=(y_1x_1)y_1^{-1}(y_1x_1)^{-1}\]

Finally
\beqn \theta^4(x_1)&=&\theta^2(y_1x_1)\theta^2(x_1^{-1})(\theta^2(y_1x_1))^{-1}=(y_1 x_1)(y_1^{-1}x_1^{-1})x_1(x_1 y_1)(y_1 x_1)^{-1}\\
&=&[x_1,y_1]^{-1}x_1 [x_1,y_1]\\[2mm]
\theta^4(y_1)&=&[x_1,y_1]^{-1}x_1 [x_1,y_1]
\eeqn

Now $[x_1,y_1]=c^2\beta_n$ so \beqn
\theta^4(x_1)&=&\beta_n^{-1}x_1\beta_n\\
\theta^4(y_1)&=&\beta_n^{-1}y_1\beta_n
\eeqn
The result now follows since \[\Delta^2=w(\sigma_2,\cdots,\sigma_{n-1})\sigma_1\sigma_2\cdots\sigma_{n-1}\sigma_{n-1}\cdots\sigma_2\sigma_1 =w\beta_n\] and the braid $w$ commutes with $x_1$ and $y_1$.
\end{proof}

\begin{remark} Simental, \cite[Lemma 2.4.20]{Simmental}, in notes which are part of a seminar series at MIT and Northeastern in 2017, makes a similar observation when applied to the spherical algebra $\SH_n$, to demonstrate the construction of the elements $Q_\xx$.
\end{remark}
We can go further and define $\tilde Q_{0,m}$ for $m>0$, by
\[\tilde Q_{0,m}=y_1^m+y_2^m+\cdots+y_n^m.\] 
Then \[Q_{0,m}=e_n\tilde Q_{0,m} e_n,\] in \cite{SV11}.
We follow the same procedure as in \cite{SV11} to define $\tilde Q_\xx$ from $\tilde Q_{0,m}$ by applying an automorphism from $SL(2,\Z)$ which takes $(0,m)$ to $\xx$.

This gives a well-defined element $\tilde Q_\xx$, provided we can show that $(\tau_1 \tau_2^{-1}\tau_1)^4=\theta^4$ acts trivially on $\tilde Q_{0,m} = \sum y_i^m$.
So we prove
\begin{lemma}
\[\Delta^{-2}(y_1^m+\cdots+y_n^m)\Delta^2 =y_1^m+\cdots+y_n^m.\]
\end{lemma}
\begin{proof}
It is enough to show that $y_1^m+\cdots+y_n^m$ commutes with $\sigma_i$ for all $i$. Now $\sigma_i$ commutes with $y_j$ for $j\ne i, i+1$. So we just need to show that $\sigma_i $ commutes with $y_i^m+y_{i+1}^m$. 

This in turn follows once we prove that
\beqn \sigma_i(y_i +y_{i+1})&=&(y_i +y_{i+1})\sigma_i\\
\sigma_i(y_i y_{i+1})&=&(y_i y_{i+1})\sigma_i
\eeqn

Now \beqn \sigma_i(y_i +y_{i+1})&=&\sigma_i y_i +\sigma_i^2y_i\sigma_i = \sigma_i y_i +y_i\sigma_i
 +(s-s^{-1})\sigma_iy_i\sigma_i\\
 & =&y_i\sigma_i  +\sigma_iy_i\sigma_i^2 = (y_i +y_{i+1})\sigma_i, \\[2mm]
 \sigma_i(y_i y_{i+1})&=&\sigma_iy_i\sigma_iy_i\sigma_i =y_{i+1}y_i\sigma_i =(y_i y_{i+1})\sigma_i.
\eeqn

This completes the proof.
\end{proof}

So we have constructed elements $\tilde Q_\xx \in \daha_n$ with $Q_\xx=e_n \tilde Q_\xx e_n$, which are related even more directly to the elements $W_\xx$ in $\sk(T^2,*)$, in the following enhancement of theorem \ref{PWcomparison}.
\begin{theorem}  For every non-zero  $\xx\in\Z^2$ we have
\[(q^m-1)W_\xx =(s^m-s^{-m})\tilde Q_\xx,\] where $\xx=m\yy$ with $\yy$ primitive and $m>0$.
\end{theorem}


\subsection{The punctured torus and elliptic Hall algebra}
In this subsection, we use the previous results in this section to show that Conjecture \ref{conj:mainiso} implies a weakened version of Conjecture \ref{conj:punctoeha}. Recall that $\ZZ^+ \subset \ZZ = \Z^2$ is defined by
\[
\ZZ^+ := \{(a,b) \in \Z^2 \mid a > 0\} \sqcup \{(0,b) \mid b \geq 0\}
\]

\begin{definition}\label{def:posskein}
Let $\sk^+(T^2\oursetminus D^2) $ be the subalgebra of $  \sk(T^2\oursetminus D^2)$  generated by $W_\xx$ for $\xx \in \ZZ^+ $.
\end{definition}

\begin{theorem}\label{thm:psktoeha}
	If Conjecture \ref{conj:mainiso} is true, then there is a surjective algebra map $\sk^+(T^2\oursetminus D^2) \twoheadrightarrow \E^+_{\sigma, \bar \sigma}$ sending $W_\xx$ to $(s^{d(\xx)} - s^{-d(\xx)}) u_\xx$.
\end{theorem}
\begin{proof}
	By Conjecture \ref{conj:mainiso}, the map $F_n:\ddot H_n \to \sk_n(T^2,*)$ is an isomorphism, and we can compose its inverse with the natural map \[\varphi_n:\sk(T^2\oursetminus D^2) \to \sk_n(T^2,*)\] to obtain a map $\sk(T^2 \oursetminus D^2) \to \ddot H_n$. By Theorem \ref{PWcomparison}, this map satisfies the following equation: 
	\begin{equation*}
	W_\xx \mapsto \frac{s^{d(\xx)} - s^{-d(\xx)}}{q^{d(\xx)} - 1} Q_\xx
	\end{equation*}
	By Corollary  \ref{cor:limitmap}, this proves the existence of the algebra map stated in the theorem, and surjectivity follows immediately from the definition of the subalgebra $\E_{\sigma, \bar \sigma}^+$.
\end{proof}

The simplest relations between the $W_\xx$ are easy to check in the elliptic Hall algebra independently of Conjecture \ref{conj:mainiso} (see Remark \ref{rmk:ptorusrel}). However, describing all relations between the $W_\xx$ in the punctured torus is an open problem.

\begin{remark}
	It would be desirable to extend the map in Theorem \ref{thm:psktoeha} to a much larger subalgebra of $\sk(T^2-D^2)$, instead of the subalgebra $\sk^+(T^2 - D^2)$ generated by $W_\xx$ for $\xx \in \ZZ^+$. 
	It seems that the main difficulty is showing compatibility between the Schiffmann-Vasserot projections between spherical DAHAs and the maps from $\sk(T^2-D^2)$ to the spherical DAHAs.
	Ideally this would follow from a topological interpretation of the Schiffmann-Vasserot projections as some kind of partial trace, but it isn't clear if such an interpretation exists. We do note that we have defined algebra maps  from the entire skein algebra $\sk(T^2 \oursetminus D^2)$ to the DAHAs (and not just the positive subalgebra). The technical difficulty here is that the Schiffmann-Vasserot projections between spherical DAHAs of different ranks are only defined on the ``positive subalgebras.''
\end{remark}


\bibliography{bibtex_dahamodels}{}
\bibliographystyle{amsalpha}

 \end{document}